\documentclass[11pt,english,reqno]{amsart}

\usepackage{amsbsy}
\usepackage{amstext}
\usepackage{amsthm}
\usepackage{amsmath}
\usepackage{color}
\usepackage{upref}
\usepackage[numbers, sort&compress]{natbib} 
\usepackage{babel}
\usepackage{verbatim}
\usepackage{refstyle}
\usepackage{mathrsfs}
\usepackage{mathtools}
\usepackage{enumitem}
\usepackage[unicode=true,pdfusetitle, bookmarks=true,bookmarksnumbered=false,bookmarksopen=false, breaklinks=true,pdfborder={0 0 0}, pdfborderstyle={},backref=false,colorlinks=false]{hyperref}

\usepackage[capitalize]{cleveref}
\crefformat{equation}{(#2#1#3)}

\numberwithin{equation}{section}
\numberwithin{figure}{section}
\theoremstyle{plain}
\newtheorem{thm}{Theorem}[section]
\crefname{thm}{Theorem}{Theorems}
\Crefname{thm}{Theorem}{Theorems}
\theoremstyle{plain}
\newtheorem{conjecture}[thm]{Conjecture}
\theoremstyle{definition}

\theoremstyle{remark}
\newtheorem{rem}[thm]{Remark}
\theoremstyle{plain}
\newtheorem{lem}[thm]{Lemma}
\theoremstyle{plain}
\newtheorem{prop}[thm]{Proposition}
\theoremstyle{plain}
\newtheorem{cor}[thm]{Corollary}

\hypersetup{final}
\usepackage{mathrsfs}
\DeclareMathAlphabet{\mathscrbf}{OMS}{mdugm}{b}{n}

\begin{document}
\global\long\def\divg{\operatorname{div}}

\global\long\def\Hess{\operatorname{Hess}}

\global\long\def\Law{\operatorname{Law}}

\global\long\def\supp{\operatorname{supp}}

\global\long\def\spn{\operatorname{span}}

\global\long\def\tr{\operatorname{tr}}

\global\long\def\Re{\operatorname{Re}}

\global\long\def\Lip{\operatorname{Lip}}

\global\long\def\Im{\operatorname{Im}}

\global\long\def\dif{\mathrm{d}}

\global\long\def\e{\mathrm{e}}

\global\long\def\ii{\mathrm{i}}

\global\long\def\Dif{\mathrm{D}}

\global\long\def\eps{\varepsilon}

\global\long\def\Cov{\operatorname{Cov}}

\global\long\def\Var{\operatorname{Var}}

\global\long\def\sgn{\operatorname{sgn}}

\global\long\def\He{\operatorname{He}}

\title[Solution of the $(2+1)$-dimensional KPZ equation]{Constructing a solution of the $(2+1)$-dimensional KPZ equation}

\author{Sourav Chatterjee}

\address{\newline Department of Statistics \newline Stanford University\newline Sequoia Hall, 390 Serra Mall \newline Stanford, CA 94305\newline  \textup{\tt souravc@stanford.edu}}

\author{Alexander Dunlap}

\address{\newline Department of Mathematics \newline Stanford University\newline 450 Serra Mall, Building 380 \newline Stanford, CA 94305\newline \textup{\tt ajdunl2@stanford.edu}}

\thanks{S.~C.~was partially supported by NSF grant DMS-1608249}
\thanks{A.~D.~was partially supported by an NSF Graduate Research Fellowship.}
\keywords{KPZ equation, Cole--Hopf transformation, renormalization}
\subjclass[2010]{60H15, 81T15, 35R60}

\begin{abstract}
  The $(d+1)$-dimensional KPZ equation is the canonical model for the growth of rough $d$-dimensional random surfaces. A deep mathematical understanding of the KPZ equation for $d=1$ has been achieved in recent years, and the case $d\ge 3$ has also seen some progress. The most physically relevant case of $d=2$, however, is not very well-understood mathematically, largely due to the renormalization that is required: in the language of renormalization group analysis, the $d=2$ case is neither ultraviolet superrenormalizable like the $d=1$ case nor infrared superrenormalizable like the $d\ge 3$ case. Moreover, unlike in $d=1$, the Cole--Hopf transform is not directly usable in $d=2$ because solutions to the multiplicative stochastic heat equation are distributions rather than functions. In this article we show the existence of subsequential scaling limits as $\eps\to0$ of Cole--Hopf solutions of the $(2+1)$-dimensional KPZ equation with white noise mollified to spatial scale $\eps$ and nonlinearity multiplied by the vanishing factor $|\log\eps|^{-\frac{1}{2}}$. We also show that the scaling limits obtained in this way do not coincide with solutions to the linearized equation, meaning that the nonlinearity has a non-vanishing effect. We thus propose our scaling limit as a  notion of KPZ evolution in $2+1$ dimensions.
\end{abstract}

\maketitle

\section{Introduction}
\subsection{Main results}
We are interested in the \emph{space-time $(2+1)$-dimensional KPZ equation on the torus},
formally given by the stochastic PDE
\begin{equation}\label{eq:formalKPZ}
\partial_{t}h=\nu\Delta h+\frac{\lambda}{2}|\nabla h|^{2}+\sqrt{D}\dot{W},
\end{equation}
where $\nu$, $\lambda$
and $D$ are strictly positive parameters and $\dot{W}$ denotes a standard space-time white noise on the two-dimensional
torus~$\mathbf{T}^{2}=\mathbf{R}^{2}/\mathbf{Z}^{2}$. More precisely, we define $W$ to be a cylindrical Wiener process on $L^2(\mathbf{T}^2)$ whose covariance operator is the identity, as in \cite{DPZ92} or \cite{BC95}, and then $\dot{W}$ is its (distributional) derivative in time. Thus, formally we have
\[
 \mathbf{E}\dot{W}(t,x)\dot{W}(t',x')=\delta(t-t')\delta(x-x').
\]
(Throughout this manuscript, we will assume that all random variables are defined on some common probability space $(\Omega, \mathcal{F}, \mathbf{P})$, and $\mathbf{E}$ will denote expectation with respect to $\mathbf{P}$.) We will use notation of the form 
\[
\int_0^t \int f(s,y)W(\dif s\,\dif y)
\] 
for $f$ integrated against the white noise. Here and throughout, an integral without a specified domain of integration will denote integration over $\mathbf{T}^2$. 

This model of interface growth was originally
introduced by Kardar, Parisi and Zhang in~\cite{KPZ86} and has subsequently
 been the subject of intense study in the physics and mathematics
literatures, especially in $1+1$ dimensions. See \cite{QS15} for
a review of mathematical results in the $1+1$-dimensional case,~\cite{CT18} for an analysis of a related discrete model in
$2+1$ dimensions, \cite{BC95, BC98, Fen15, CSZ17, CSZ18} for results on the multiplicative
stochastic heat equation in $2+1$ dimensions with implications for
KPZ, and \cite{MU18} for some recent progress about the $(d+1)$-dimensional equation for $d\ge 3$. A more extensive discussion of the literature is given in \cref{subsec:litreview}. We discuss the paper \cite{CSZ17} further in \cref{sec:CH}.

Defining solutions to \cref{eq:formalKPZ} is a well-known challenge
in the theory of stochastic PDEs, because the roughness of the driving
noise $\dot{W}$ precludes the existence of solutions smooth enough
for the nonlinear term $|\nabla h|^{2}$ to have meaning. The usual
approach is to proceed by \emph{mollification} of the noise $\dot{W}$,
in space and sometimes also in time, and then attempting to take a
limit as the mollifier approaches a delta function. Implementing this
strategy requires some form of \emph{renormalization} --- subtracting
divergent counterterms and/or modifying the parameters of the equation --- 
in a manner that gives rise to a scaling limit as the mollifier approaches a delta function.

In this paper we propose a renormalization scheme for \cref{eq:formalKPZ}
in $2+1$ dimensions, and show that subsequential limits
of the solutions exist as the mollification is turned off. Moreover, we show that the limiting solutions are not the same as the the limiting solution to the same sequence of equations with no nonlinear term, so the nonlinearity has a non-vanishing effect.  In order
to state our main theorems, we need to introduce some notation. 
Let $\rho\in\mathcal{C}^{\infty}(\mathbf{R}^{2})$ be a positive even
function so that 
\[
\supp\rho\subset\biggl(-\frac{1}{2},\frac{1}{2}\biggr)^{2}
\]
and $\|\rho\|_{L^{1}}=1$, and define, for $\eps\in(0,1)$, $\rho^{\eps}(x)=\eps^{-2}\rho(\eps^{-1}x).$
Thus $\rho^\eps$ descends trivially to a function in $\mathcal{C}^{\infty}(\mathbf{T}^{2})$ by periodic extension,
which we will identify with $\rho^\eps$. We then define the mollified
white noise as the $\mathbf{T}^2$-convolution
\begin{equation}
W^{\eps}=\rho^{\eps}*W.\label{eq:mollification}
\end{equation}
Let $\widetilde{h}^\eps$ be a solution of \cref{eq:formalKPZ} with $\dot{W}$ replaced by $\dot{W}^\eps$. As we show in \cref{sec:CH} below, such a solution exists because of the spatial smoothness of  $W^\eps$. Our primary goal is to understand the behavior of $h^\eps$ as $\eps\to 0$. For the $(1+1)$-dimensional KPZ equation, this limiting behavior is well understood; simply subtracting off a deterministic function of $t$ and $\eps$ gives a nontrivial scaling limit. For the $(2+1)$-dimensional equation, we will show that a nontrivial scaling limit can be obtained if we renormalize the nonlinearity parameter $\lambda$  by replacing it with $\lambda |\log \eps|^{-\frac{1}{2}}$. To be precise, let $\widetilde{h}^\eps$ solve
\begin{equation}
\begin{cases}
\partial_{t}\widetilde{h}^{\eps}(t,x)=\nu\Delta\widetilde{h}^{\eps}(t,x)+\frac{1}{2}\lambda|\log\eps|^{-\frac{1}{2}}|\nabla\widetilde{h}^{\eps}(t,x)|^{2}&\\
\qquad \qquad \qquad +\sqrt{D}\dot{W}^{\eps}(t,x),& t>0,x\in\mathbf{T}^{2},\\
\widetilde{h}^{\eps}(0,x)=0, & x\in\mathbf{T}^{2},
\end{cases}\label{eq:htildePDE0}
\end{equation}
and then let 
\begin{align}\label{eq:hepsdef}
h^\eps(t,x) := \widetilde{h}^\eps(t,x) - \kappa^\eps(t),
\end{align}
where  $\kappa^\eps(t) := \mathbf{E}\widetilde{h}^\eps(t,x)$ is a deterministic quantity that depends only on $t$ and not on $x$.  We will say more about $\kappa^\eps(t)$ later.

As noted above, the roughness of the driving noise $\dot{W}$ means
that we do not expect the limits of $h^{\eps}$ as $\eps\to0$
to be smooth. In fact, unlike the one-dimensional case, the limits here are not even functions. Rather, the limits exist in spaces of distributions.
We will prove tightness in certain negative H\"older spaces which we
will introduce in \cref{sec:negholder}; see \cref{thm:maintheorem-quantitative}
in that section. Since we do not expect that the regularity we achieve there
is optimal, for now we state the following simpler corollary.
Recall the spaces of distributions
\[
\mathcal{D}'(\mathbf{T}^{2})=\mathcal{C}^{\infty}(\mathbf{T}^{2})^{*}
\]
and
\[
\mathcal{D}'(\mathbf{R}_{>0}\times\mathbf{T}^{2})=\mathcal{C}_{c}^{\infty}(\mathbf{R}_{>0}\times\mathbf{T}^{2})^{*},
\]
where the asterisks denote the Fr\'echet space duals. We note in particular
in the second definition that $\mathcal{C}_{c}^{\infty}(\mathbf{R}_{>0}\times\mathbf{T}^{2})$
means the space of smooth functions supported on a compact subset
of $\mathbf{R}_{>0}\times\mathbf{T}^{2}$, so in particular with support
bounded away from zero. The following theorem, which is our first main result,  establishes the existence of subsequential scaling limits. 
\begin{thm}
\label{thm:maintheorem0}
Let $h^\eps$ be defined as in $\cref{eq:hepsdef}$ above. There is a $\theta_{0}>0$ so that if 
\begin{equation}
 \frac{\lambda^2 D}{(2\nu)^{3}}\le \theta_0,\label{eq:condition}
\end{equation}
then the following hold. For any sequence $\eps_{n}\downarrow0$,
there is a subsequence $\eps_{k_{n}}$ and a $\mathcal{D}'(\mathbf{R}_{>0}\times\mathbf{T}^{2})$-valued random distribution $h$ such that $h^{\eps_{k_{n}}}\to h$ in law as $n\to \infty$. Moreover, for any sequence $\eps_{n}\downarrow0$ and any $t>0$, there is a subsequence $\eps_{k_{n}}$ and a 
$\mathcal{D}'(\mathbf{T}^{2})$-valued random distribution $h_t$ such that $h^{\eps_{k_n}}(t,\cdot)\to h_{t}$ in law as $n\to \infty$.
\end{thm}
\begin{rem}
 Since the space of distributions is not metrizable, the usual Portmanteau lemma for weak convergence of measures does not apply. The notion of weak convergence being used in \cref{thm:maintheorem0} is that of convergence of the expectation of every bounded continuous functional, and the $\sigma$-algebra used to define the notion of random distributions is the $\sigma$-algebra generated by the weak-* topology. However, the convergence actually takes place in a local H\"older space with negative regularity exponent, which is a Fr\'echet space. We postpone the stronger statement (\cref{thm:maintheorem-quantitative}) until after we have introduced the necessary definitions of negative H\"older regularity in \cref{sec:negholder}.
\end{rem}
\begin{rem}
The quantity appearing on the left side of \cref{eq:condition} is called the {\it effective coupling constant} in the physics literature on the renormalization group for the KPZ equation~\cite{CCDW10}. It measures how strongly the nonlinearity is coupled to the linear system. In physics  terminology, \cref{thm:maintheorem0} would be called a weak coupling result.
\end{rem}
\begin{rem}
It is important to understand whether the condition \cref{eq:condition}  is necessary in the statement of \cref{thm:maintheorem0}. In light of recent results about the $(2+1)$-dimensional stochastic heat equation with multiplicative noise~\cite{Fen15, CSZ17, CSZ18}, we believe  that there are values of $\nu$, $\lambda$ and $D$ for which \cref{thm:maintheorem0} is not valid (specifically, when the effective coupling constant is large). Understanding this is at present out of the reach of the technology developed in this paper.
\end{rem}
\begin{rem}The two subsequential convergences stated in \cref{thm:maintheorem0} could be unified into a single statement if we could show that the convergence holds in some space of continuous maps from $\mathbf{R}_{\ge 0}$ into $\mathcal{D}'(\mathbf{T}^2)$, endowed with some topology that is strong enough to at least render the pointwise evaluation maps continuous. We expect this, but are currently unable to prove it.
\end{rem}

We are at this point unable to show that the subsequential scaling limits are unique. Seeing
no reason for them not to be unique, however, we state the following conjecture.
\begin{conjecture}
Under the condition $\cref{eq:condition}$, the sequences $(h^\eps)_{\eps>0}$ and $(h^\eps(t,\cdot))_{\eps>0}$ converge in law as $\eps\to 0$.
\end{conjecture}

In order to defend our choice of scaling for the nonlinearity parameter in \cref{eq:htildePDE0} as
an interesting notion of the KPZ evolution in $2+1$ dimensions, we
need to show that the limits we establish in \cref{thm:maintheorem0}
exhibit a non-vanishing effect of the nonlinearity. Indeed, \emph{a priori}, we might worry that the coefficient
$|\log\eps|^{-\frac{1}{2}}$ is going to zero so quickly
with $\eps$ that any subsequential scaling limit $h$ has the same law as a solution of the additive
stochastic heat evolution with the same noise strength and diffusivity,
given by taking $\lambda=0$ in \cref{eq:formalKPZ}. By this we mean the distribution
$v$ solving the problem
\begin{equation}
\begin{cases}
\partial_{t}v(t,x)=\nu\Delta v(t,x)+\sqrt{D}\dot{W}(t,x), & t>0,x\in\mathbf{T}^{2},\\
v(0,x)=0, & x\in\mathbf{T}^{2}.
\end{cases}\label{eq:htildePDE-10}
\end{equation}
(Solutions to \cref{eq:htildePDE-10} are given by the Green's function of the heat equation convolved with the white noise.)
In our second theorem, stated below, we prove that this does not happen. The theorem further shows that the scaling limit is not a constant field, nor is it a constant shift of the solution of \cref{eq:htildePDE-10}.
\begin{thm}
\label{thm:limitsnontrivial0}
Take any $t>0$. Suppose that $h^\eps(t,\cdot)$ converges in law to some limit $h_t$ through a subsequence as $\eps\to 0$. Let $v$ be a solution of  $($\ref{eq:htildePDE-10}$)$. Then $\int h_t(x)\,\dif x$ and  $\int v(t,x)\,\dif x$ are both non-degenerate random variables with mean zero, but their laws are different.
\end{thm}
Since the nonlinearity has a nontrivial effect when the nonlinearity parameter is scaled like a multiple of $|\log \eps|^{-\frac{1}{2}}$, it seems unlikely and unnatural that a nontrivial scaling limit can be obtained by some other (faster or slower) scaling of the nonlinearity parameter. With this intuition in mind, we make the following conjecture.
\begin{conjecture}
The only way to obtain a nontrivial scaling limit in $\cref{eq:formalKPZ}$ with $\nu$ and $D$ fixed is to scale the nonlinearity parameter $\lambda$ like a multiple of $|\log \eps|^{-\frac{1}{2}}$, as we did in $\cref{eq:htildePDE0}$.
\end{conjecture}

There is some indirect evidence for this conjecture from the existing analysis of the $(2+1)$-dimensional stochastic heat equation with multiplicative noise. We discuss this and other connections with the literature below.

\begin{rem}\label{rem:otherscalings}
It is possible that more general scalings of the parameters (that is, with $\nu$ and/or $D$ allowed to vary with $\eps$) could lead  to other scaling limits. Numerical simulations such as \cite{KO11,RMO14} suggest that it may also be possible to obtain a \emph{function}-valued scaling limit by taking (in what is known in physics as the \emph{Family--Vicsek scaling}~\cite{FV91}) $\nu\sim\eps^{2-z}$, $\lambda\sim\eps^{2-z-\alpha}$, and $D\sim\eps^{2+2\alpha-z}$ for certain particular exponents $\alpha$ and $z$, which scaling arguments based on Galilean invariance~\cite{BS95} suggest should satisfy $\alpha+z=2$. If we assume this, then we obtain the scaling $\nu\sim\eps^\alpha$, $\lambda\sim 1$, and $D\sim\eps^{3\alpha}$. By an analogue of the change of variables described in \cref{sec:onepara} below, this amounts to considering \cref{eq:formalKPZ} with \emph{fixed} values of the parameters, and considering the solution multiplied by $\eps^\alpha$ on a short time scale $t\sim\eps^{\alpha}$. We do not consider this setting further in this paper.
 \end{rem}

\subsection{Comparison with the literature}\label{subsec:litreview}
The literature surrounding the KPZ equation has grown exponentially in the last few years, and keeps growing each day. It is quite impossible to review (even briefly) all of the developments within one section of a paper. Here we only survey the part of the literature that is closest to this paper, and compare our results with the existing ones. 

As mentioned before, our scaling of the nonlinearity parameter is in contrast to the 
results of \cite{BC14, BG97,ACQ11,Hai13,GP17,K17} in $1+1$ dimensions, in
which there is a diverging renormalization constant $\kappa^{\eps}(t)\sim\eps^{-1}t$,
but the parameters $\nu,\lambda, D$ are all kept fixed as $\eps\to0$. The limit object is an actual random function, not a distribution. In fact, it is the logarithm of a solution of the stochastic heat equation (SHE) with multiplicative noise. The main difficulty about extending this approach to $2+1$ dimensions is that the solutions of the $(2+1)$-dimensional SHE with multiplicative noise (which are now fairly well-understood, thanks to~\cite{BC95, BC98, Fen15, CSZ17, CSZ18}) are random distributions instead of random functions, and we do not know how to take logarithms of distributions.

The KPZ equation in the $(1+1)$-dimensional case is, in the mathematician's language, \emph{locally subcritical}~\cite{H14}, or, in the physicist's language, \emph{ultraviolet superrenormalizable}~\cite{MU18}, meaning that when a parabolic
scaling is applied to \cref{eq:formalKPZ}, the coefficient in front
of the nonlinearity disappears compared to that of the noise and the
Laplacian on very small scales.  A similar phenomenon, known as \emph{infrared superrenormalizability}~\cite{MU18}, happens under a different rescaling for the $(d+1)$-dimensional KPZ equation when $d\ge 3$. 

In contrast, the $(2+1)$-dimensional case is critical in that any rescaling leaves the nonlinearity with a non-vanishing coefficient. In physics language, the $(2+1)$-dimensional KPZ equation is neither ultraviolet nor infrared superrenormalizable. The lack of superrenormalizability means that the methods of regularity
structures~\cite{Hai13,H14}, paracontrolled distributions~\cite{GIP15} and constructive field theory~\cite{MU18}   
do not apply to the $(2+1)$-dimensional KPZ equation. Our results show that it is, however, \emph{renormalizable} if we reduce the strength of the nonlinearity logarithmically as the mollification is sent to zero.

Such a logarithmic scaling has a famous precedent in the Nobel prize-winning papers~\cite{GW73, P73}, where it was shown that for the renormalization of four-dimensional non-Abelian gauge theories, the bare coupling constant should vanish logarithmically as the ultraviolet cutoff is removed (see \cite{MW00} for a friendly explanation). In that case, as in ours,  a na\"ive dimension counting argument suggests that the coupling constant should not be scaled at all, but this is shown to be wrong by renormalization group analysis. This is very similar in spirit to our results for the $(2+1)$-dimensional KPZ equation. In fact, the scaling is also exactly the same, namely, if $\eps$ is the lattice spacing in a lattice regularized gauge theory, then the coupling constant should scale like a multiple of $|\log \eps|^{-\frac{1}{2}}$ as $\eps\to 0$. The results for gauge theories, however, have not yet been made mathematically rigorous.

Our mollification of the noise in \cref{eq:mollification} is only spatial,
so the noise remains white in time in our approximation scheme. This
mollification scheme has been used in the past when using the Cole--Hopf
transform in $1+1$ dimensions \cite{BG97,Hai13}. On the other hand,
it is certainly not the only physically relevant mollification scheme;
for example, one could use a space-time mollification as \cite{MU18}
does in three space dimensions --- see the next paragraph. Because of
the roughness of the problem, there is no reason to expect that different
approximation schemes lead to the same scaling limits. (This is true
even in stochastic ordinary differential equations, in which different
approximation schemes can lead to the difference between the It\^{o} and
Stratonovich integrals.) The recent theories of \cite{Hai13,HMW14,H14,GIP15,K16}
have made substantial progress towards interpreting solutions of (locally
subcritical) singular stochastic PDEs in such a way that the effect
of the choice of the approximation scheme on the scaling limit can
be understood. Since our present work does not even show that the
scaling limit is unique for our \emph{single} choice of mollification,
we are definitely not yet in a position to understand the effect of
the approximation on the scaling limit in $2+1$ dimensions.

We note that our situation is similar to the work \cite{MU18} in
dimensions $d\ge3$, which obtained a scaling limit with $\lambda=\eps^{\frac{d}{2}-1}$
(so also attenuating the nonlinearity as $\eps\to0$, see \cite{GRZ18}
for further discussion) and showed that it satisfies an additive
stochastic heat equation with a \emph{modified} effective diffusivity
and noise strength, that is, a different choice of $\nu$ and $D$. However,
\cite{MU18} considers noise that is mollified in both time and space,
so the correct analogy with our white-in-time case could be that only
the noise strength should be modified --- this is the situation for
the multiplicative stochastic heat equation in $d\ge3$ \cite{MSZ16,GRZ18,M17}. 

 The possibility of a Gaussian scaling limit, as in \cite{MU18}, is \emph{not} ruled out by our \cref{thm:limitsnontrivial}.  
Indeed, it is quite possible that the scaling limits we obtain in
our setting are Gaussian, especially in view of the Gaussian limits
obtained by \cite{Fen15, CSZ17, CSZ18} for the multiplicative stochastic heat equation
in $(2+1)$-dimensions, which we discuss in more detail later.

\subsection{Proof strategy}
Our proof strategy for \cref{thm:maintheorem0} is inspired by intuition from perturbative renormalization, combined with probabilistic techniques. We use the Gaussian Poincar\'e inequality, together with the Cole--Hopf transformation, the Feynman--Kac formula, and recently derived tightness criteria for negative H\"older spaces~\cite{FM17, G18}, to conclude that $(h^\eps)_{\eps>0}$ is a tight family of random distributions if the expected value of a certain intersection local time under a randomly tilted Wiener measure remains finite as $\eps\to 0$. The tilting involves the nonlinearity parameter $\lambda$, the mollification parameter $\eps$, and the white noise $\dot{W}$. To understand this  expected value, it is natural to try to expand it as a power series in $\lambda$. This resembles the expansions commonly occurring in perturbative renormalization. In fact, if $a_k$ is the coefficient of $\lambda^k$, then $a_k$ can, in principle, be written using Feynman diagrams, since we are expanding around $\lambda=0$, which corresponds to a Gaussian measure.  

However, understanding these coefficients is likely to be a very complex task, intimately tied to the complexities of the so far unsolved task of renormalizing the $(2+1)$-dimensional KPZ equation. Instead, we adopt a different strategy, which can be roughly described as follows. If $f_\eps(\lambda)$ is the original function of $\lambda$ that we are trying to bound, then we first bound it by a simpler function $g_\eps(\lambda)$. Then we exhibit a sequence of nonnegative functions $\{g_{\eps,k}\}_{k\ge 0}$, with $g_{\eps,0}=g_\eps$, such that they satisfy a hierarchical system of differential inequalities of the form $|g_{\eps,k}'(\lambda)|\le Cg_{\eps,k+1}(\lambda)$, where $C$ is some constant that does not depend on $k$ or $\eps$. An $\eps$-free bound on $f_\eps(\lambda)$ is then obtained by manipulating this hierarchical system of inequalities.

Our proof of \cref{thm:limitsnontrivial0} relies on the observation that $\int v(t,x)\,\dif x$ depends only on the spatial averages of the white noise over the entire torus. These spatial averages, having no spatial fluctuations and thus not feeling any effect of the nonlinearity, have exactly the same effect on $\int h^{\eps}(t,x)\,\dif x$ as they do on $\int v(t,x)\,\dif x$. However, we will show in \cref{sec:nontriviality} that $\int h^{\eps}(t,x)\,\dif x$, in contrast to $\int v(t,x)\,\dif x$, also feels effects of higher Fourier modes of the white noise, to an extent that does not diminish as $\eps\to0$. Calculating and understanding the higher Fourier modes involves Malliavin calculus and hypercontractivity, along with the inequalities described in the previous paragraph.


\subsection{Organization of the paper}
The rest of the paper is organized as follows. In \cref{sec:onepara} we reduce the number of parameters from three to one by a suitable rescaling of the equation. In \cref{sec:CH} we introduce the Cole--Hopf transformation and the Feynman--Kac representation of the solutions to the mollified equation. In \cref{sec:negholder} we introduce negative H\"older spaces and criteria for tightness of probability measures on such spaces. In \cref{sec:gaussian} we  recall some basic facts about Malliavin calculus that we will use.  In \cref{sec:derivatives}
we establish key derivative formulas that we will use throughout the
paper. \cref{sec:Taylorbound},
the heart of the work, proves the convergence of the infinite series mentioned above. 
We conclude the proof of \cref{thm:maintheorem0} in \cref{sec:maintheoremproof}.
Finally, in \cref{sec:nontriviality} we prove \cref{thm:limitsnontrivial0}.
Several of our estimates involve somewhat lengthy but straightforward
calculations, which we defer to \cref{sec:techproofs} to preserve the
flow of the main arguments.

\subsection{Acknowledgments}

We thank Felipe Hernandez and Lenya Ryzhik for helpful
conversations. We are also grateful to Ivan Corwin, Martin Hairer, Jean-Christophe Mourrat, and Xianliang Zhao for insightful feedback on a draft of the paper; in particular, we thank Martin Hairer for bringing to our attention the conjectures mentioned in \cref{rem:otherscalings}. Finally, the comments of two anonymous referees helped improve the presentation in numerous places.

\section{Reduction to one parameter}\label{sec:onepara}
Let $\widetilde{h}^\eps$ be a solution to \cref{eq:htildePDE0}. Let us rescale $\widetilde{h}^\eps$ by defining a new process
\[
\widetilde{g}^\eps(t,x) := \lambda(2\nu)^{-1}\widetilde{h}^\eps((2\nu)^{-1} t, x).
\]
An easy verification shows that $\widetilde{g}^\eps$ satisfies the equation
\begin{align*}
\partial_t \widetilde{g}^\eps(t,x) &= \frac{1}{2}\Delta \widetilde{g}^\eps(t,x) + \frac{1}{2}|\log \eps|^{-\frac{1}{2}} |\nabla \widetilde{g}^\eps(t,x)|^2 \\
&\qquad + \lambda(2\nu)^{-2}\sqrt{D}\dot{W}^\eps((2\nu)^{-1}t,x),
\end{align*}
with $\widetilde{g}^\eps(0,x)=0$. 
Now let
\[
B(t,x) := (2\nu)^{1/2}W((2\nu)^{-1}t,x).
\]
Then $\dot{B}$ is again a standard space-time white noise. Let $B^\eps := \rho^\eps* B$. Since there is no scaling in space in the definition of $B$, it follows that
\[
B^\eps(t,x) := (2\nu)^{1/2}W^\eps((2\nu)^{-1}t,x).
\]
Thus, $\widetilde{g}^\eps$ satisfies the equation
\begin{align*}
\partial_t \widetilde{g}^\eps(t,x) &= \frac{1}{2}\Delta \widetilde{g}^\eps(t,x) + \frac{1}{2}|\log \eps|^{-\frac{1}{2}} |\nabla \widetilde{g}^\eps(t,x)|^2 +\sqrt{\theta}\dot{B}^\eps(t,x),
\end{align*}
where
\begin{equation}\label{eq:thetadef}
\theta := \frac{\lambda^2 D}{(2\nu)^3}.
\end{equation}
Therefore, to study \cref{eq:htildePDE0}, it suffices to study the following  stochastic PDE, which involves only one positive parameter $\theta$:
\begin{equation}
\begin{cases}
\partial_{t}\widetilde{h}_{\theta}^{\eps}(t,x)=\frac{1}{2}\Delta\widetilde{h}_{\theta}^{\eps}(t,x)+\frac{1}{2}|\log\eps|^{-\frac{1}{2}}|\nabla\widetilde{h}_{\theta}^{\eps}(t,x)|^{2}&\\
\qquad \qquad \qquad +\sqrt{\theta}\dot{W}^{\eps}(t,x),& t>0,x\in\mathbf{T}^{2},\\
\widetilde{h}_{\theta}^{\eps}(0,x)=0, & x\in\mathbf{T}^{2}.
\end{cases}\label{eq:htildePDE}
\end{equation}
Further, we define
\begin{equation}
h_{\theta}^{\eps}(t,x)=\widetilde{h}_{\theta}^{\eps}(t,x)-\kappa_{\theta}^{\eps}(t),\label{eq:renormalization}
\end{equation}
where
\begin{equation}
\kappa^{\eps}_\theta(t)=\mathbf{E}\widetilde{h}_{\theta}^{\eps}(t,x).\label{eq:renormalizationconstantexpdef}
\end{equation}
With the above definitions, the relation \cref{eq:thetadef} shows that \cref{thm:maintheorem0} is equivalent to the following theorem about $h^\eps_\theta$.
\begin{thm}
\label{thm:maintheorem}There is a $\theta_{0}>0$ so that if $\theta\le\theta_0$, then the following hold. For any sequence $\eps_{n}\downarrow0$,
there is a subsequence $\eps_{k_{n}}$ and a $\mathcal{D}'(\mathbf{R}_{>0}\times\mathbf{T}^{2})$-valued random distribution $h_\theta$ such that $h^{\eps_{k_{n}}}_\theta \to h_\theta$ in law as $n\to \infty$. Moreover, for any sequence $\eps_{n}\downarrow0$ and any $t>0$, there is a subsequence $\eps_{k_{n}}$ and a 
$\mathcal{D}'(\mathbf{T}^{2})$-valued random distribution $h_{t;\theta}$ such that $h^{\eps_{k_n}}_\theta(t,\cdot)\to h_{t;\theta}$ in law as $n\to \infty$.
\end{thm}

We give a formula for $\kappa_{\theta}^{\eps}(t)$ in \cref{lem:expcomp},
and we  obtain the first-order asymptotics
\begin{equation}
\kappa_{\theta}^{\eps}(t)=\frac{1}{2}|\log\eps|^{-\frac{1}{2}}\frac{\theta t}{\eps^{2}}\|\rho\|_{L^{2}}^{2}+O(|\log\eps|^{\frac{1}{2}})\label{eq:kappaasympt}
\end{equation}
as $\eps\to0$ for fixed $\theta$ sufficiently small and fixed  $t$ in 
 \cref{lem:kappaasympt}. We note that the big-$O$ term in~\cref{eq:kappaasympt}
is still diverging as $\eps\to0$; understanding more precise asymptotics
of $\kappa_{\theta}^{\eps}(t)$ remains an open problem.
Next, let $v_\theta$ be a solution to the stochastic heat equation 
\begin{equation}
\begin{cases}
\partial_{t}v_{\theta}(t,x)=\frac{1}{2}\Delta v_{\theta}(t,x)+\sqrt{\theta}\dot{W}(t,x) & t>0,x\in\mathbf{T}^{2}\\
v_{\theta}(0,x)=0. & x\in\mathbf{T}^{2}
\end{cases}\label{eq:htildePDE-1}
\end{equation}
The following theorem is equivalent to \cref{thm:limitsnontrivial0}. 
\begin{thm}
\label{thm:limitsnontrivial}
Take any $t>0$. Suppose that $h^\eps_\theta(t,\cdot)$ converges in law to some limit $h_{t;\theta}$ through a subsequence as $\eps\to 0$. Let $v_\theta$ be a solution of $($\ref{eq:htildePDE-1}$)$. Then $\int h_{t;\theta}(x)\,\dif x$ and  $\int v_\theta(t,x)\,\dif x$ are both non-degenerate random variables with mean zero, but their laws are different.
\end{thm}
Throughout the rest of the paper, we will work with the processes $\widetilde{h}^\eps_\theta$ and $h^\eps_\theta$ defined here instead of the processes $\widetilde{h}^\eps$ and $h^\eps$ defined earlier. We will prove  \cref{thm:maintheorem} and \cref{thm:limitsnontrivial} instead of \cref{thm:maintheorem0} and \cref{thm:limitsnontrivial0}.


\section{The Feynman--Kac formula\label{sec:CH}}

We construct solutions to the approximating problems \cref{eq:htildePDE}
by using the Cole--Hopf transform \cite{Hop50,Col51} to transform
the equation into a multiplicative stochastic heat equation, and then
the Feynman--Kac formula to represent the solutions to the multiplicative
stochastic heat equation in terms of the expectation of a functional
of a Brownian motion. This Feynman--Kac representation will then
form the basis for our analysis throughout the paper.

The Cole--Hopf transform of $\widetilde{h}_{\theta}^{\eps}$ is defined as
\begin{equation}
\widetilde{u}_{\theta}^{\eps}=\exp\{|\log\eps|^{-\frac{1}{2}}\widetilde{h}_{\theta}^{\eps}\}.\label{eq:uexph}
\end{equation}
Using It\^{o}'s formula, it is easy to verify that this function solves the multiplicative stochastic heat equation
\begin{equation}
\begin{cases}
\partial_{t}\widetilde{u}_{\theta}^{\eps}(t,x)=\frac{1}{2}\Delta\widetilde{u}_{\theta}^{\eps}(t,x)&\\
\qquad \qquad \qquad +\frac{1}{2}|\log\eps|^{-\frac{1}{2}}\sqrt{\theta}\widetilde{u}_{\theta}^{\eps}(t,x)\dot{W}^{\eps}(t,x)& \\
\qquad \qquad \qquad +\frac{\theta}{2}|\log\eps|^{-1}\eps^{-2}\|\rho\|_{L^{2}}^{2}\widetilde{u}_{\theta}^{\eps}(t,x), & t>0,x\in\mathbf{T}^{2},\\
\widetilde{u}_{\theta}^{\eps}(0,x)=1, & x\in\mathbf{T}^{2},
\end{cases}\label{eq:MSHE}
\end{equation}
where the last term comes from the It\^{o} correction. Here, we see that
the noise has been attenuated by the same factor $|\log\eps|^{-\frac{1}{2}}$
that multiplied the nonlinearity in the KPZ equation. 

The multiplicative stochastic heat equation in $2+1$ dimensions with this noise strength has been studied in the paper \cite{CSZ17}, which showed among many other things that, for $\theta$ below a critical value, $|\log\eps|^{\frac{1}{2}}(\widetilde{u}_{\theta}^{\eps}(t,x)-1)$, averaged over a macroscopic scale, converges to a nontrivial Gaussian random variable (see \cite[Theorem 2.17]{CSZ17}). This is reminiscent of the setting of \cref{thm:maintheorem}, except that instead of subtracting $1$, we take a logarithm before multiplying $\widetilde{u}_{\theta}^{\eps}(t,x)$ by $|\log\eps|^{\frac{1}{2}}$. Because the limiting random field is a distribution rather than a function, it is not clear how to relate these results. Also, much earlier, \cite{BC98} considered a version of
\cref{eq:MSHE} with (in our notation) a very specific tuning of $\theta$
around the critical value, and showed the existence of a limit of
the covariance structure. See also \cite[Remark 2.19]{CSZ17} for
a more detailed discussion of \cite{BC98}.

Here and throughout the rest of the paper, let $\mathbb{E}_{X^{t,x}}$ denote expectation with respect to a Brownian motion on the torus, running \emph{backwards} in time, starting at position $x\in\mathbf{T}^{2}$ at time $t\ge0$.  By the generalized Feynman--Kac formula proved in \cite{BC95}, the solution to \cref{eq:MSHE} can be written as
\begin{equation}
\widetilde{u}_{\theta}^{\eps}(t,x)=\mathbb{E}_{X^{t,x}}\exp\left\{ \theta^{\frac{1}{2}}|\log\eps|^{-\frac{1}{2}}\int_{0}^{t}\int\rho^{\eps}(X(s)-y)W(\dif y\,\dif s)\right\} .\label{eq:FK1}
\end{equation}
The proof of \cref{eq:FK1} given in \cite[(3.22)]{BC95} is for the $(1+1)$-dimensional case on the whole space; however, no part of their proof is specific to one space dimension, and replacing the white noise in \cite{BC95} with a spatially-periodic white noise (which is equivalent to working on the torus) requires no modification. The computation
of \cref{eq:FK1} previously appeared in \cite[Remark 2.16]{CSZ17};
see also the $(d+1)$-dimensional case, $d\ge3$, in \cite{MSZ16}.
Using \cref{eq:uexph}, we thus get the formula
\begin{align}
\widetilde{h}_{\theta}^{\eps}(t,x)&=|\log\eps|^{\frac{1}{2}}\log\mathbb{E}_{X^{t,x}}\exp\biggl\{ \theta^{\frac{1}{2}}|\log\eps|^{-\frac{1}{2}}\notag\\
&\qquad \qquad \qquad \qquad \cdot\int_{0}^{t}\int\rho^{\eps}(X(s)-y)W(\dif y\,\dif s)\biggr\},\label{eq:fkform}
\end{align}
and hence, by \cref{eq:renormalization},
\begin{equation}
h_{\theta}^{\eps}(t,x)=|\log\eps|^{\frac{1}{2}}\log\mathbb{E}_{X^{t,x}}\mathscr{E}^\eps_{t,\theta}[W,X],\label{eq:MainFeynmanKacFormula}
\end{equation}
where
\begin{align*}
\mathscr{E}^\eps_{t,\theta}[W,X]&=\exp\biggl\{ \theta^{\frac{1}{2}}|\log\eps|^{-\frac{1}{2}}\int_{0}^{t}\int\rho^{\eps}(X(s)-y)W(\dif y\,\dif s)\notag \\
&\qquad \qquad \qquad -|\log\eps|^{-\frac{1}{2}}\kappa_{\theta}^{\eps}(t)\biggr\},
\end{align*}
and $\kappa_\theta^\eps(t)$ is the function defined in \cref{eq:renormalizationconstantexpdef}.

The above formulas show that our model is very closely related to the directed polymer model. Indeed,  if the Brownian  motion is replaced by a random walk on a lattice and the white noise by a collection of i.i.d.~Gaussian random variables on the lattice, then the expectation in \cref{eq:MainFeynmanKacFormula} is proportional to the partition function of the directed polymer in $2+1$ dimensions. In fact, the analysis performed in this paper could equally well be done in that discrete setting, with minimal modifications to account
for the discretization. See \cite{comets17, Fen15} for other recent results
about the $(2+1)$-dimensional directed polymer model.

\section{A criterion for tightness\label{sec:negholder}}
In this section we introduce the negative H\"older spaces which we use to state \cref{thm:maintheorem-quantitative}, a stronger version of \cref{thm:maintheorem}.  
In our proof of \cref{thm:maintheorem-quantitative}, we will use the tightness criterion
for random distributions given in~\cite{FM17}. Since we will be partially
working in the parabolic setting, we will use the easy adaptation
of their results to the parabolic scaling. This adaptation was previously
stated and used in \cite{G18}. Here we only state the
results that we use in this paper.

Throughout this paper we will let $\mathfrak{s}=(2,1,1)$, corresponding
to the parabolic scaling of space-time. We first recall the definition
of negative H\"older spaces that we will use. These spaces are spaces of distributions: unlike functions, they do not take values at points, but yield values when averaged against test functions. Of course, if a test function is scaled so as to approach a delta function, the value of a distribution averaged against the test function is liable to blow up.  Unlike the spaces $\mathcal{D}'$ of distributions defined for the statement of \cref{thm:maintheorem0}, negative H\"older spaces include information about the quantitative rate of blowup as a test function is scaled as to approach a delta. See for example \cite{FM17}, \cite{H14}, or \cite{HL15} for more details on these spaces. 

Note, however, that we define the \emph{separable} versions of these spaces below, which are slightly different from the more common definitions given in \cite{H14,HL15}. The difference is that the separable versions of the spaces are the closure of $\mathcal{C}^\infty$ in the relevant norm, whereas the usual definition is simply all distributions for which the norm is finite. In our context, the distinction is not very material, because the non-separable spaces embed into the separable versions with any strictly smaller regularity exponent. However, because we are establishing a tightness result, we will want to work in separable spaces so that Prokhorov's theorem applies.

We recall the definition, for $r\in\mathbf{Z}_{\ge0}$, of $\mathcal{C}^r(U)$ to be the space of $r$-times differentiable functions on a space $U$, with the norm given by the sum of the $\mathcal{L}^\infty$ norms of the function and its derivatives up to order $r$. 
Now let $\alpha<0$, $r_{0}=-\lfloor\alpha\rfloor$. First, we will define the relevant H\"older space for functions on $\mathbf{T}^2$, which in our setting will represent the evolution at a fixed time. Let $B(0,1/2)=\{x\in\mathbf{R}^2:|x|<1/2\}$. For $\eta:\mathbf{R}^2\to\mathbf{R}$ with support contained in $B(0,1/2)$, define
\[
 \mathcal{S}^\lambda\eta(x)=\lambda^{-2}\eta(\lambda^{-1}x).
\]
Interpret $\mathcal{S}^\lambda\eta$ as a function on $\mathbf{T}^2$ by periodization. Let $\mathcal{C}^{\alpha}(\mathbf{T}^{2})$ be the completion of $\mathcal{C}^{\infty}(\mathbf{T}^{2})$ under the norm
\begin{align*}
\|f\|_{\mathcal{C}^{\alpha}(\mathbf{T}^{2})}&=\sup\biggl\{\lambda^{-\alpha}\int f(x)\mathcal{S}^\lambda\eta(x-y)\,\dif x: 0<\lambda<1,\\
&\qquad  \qquad \qquad y\in\mathbf{T}^{2}, \, \eta\in\mathcal{C}_{c}^{r_{0}}(B(0,1/2)), \, \|\eta\|_{\mathcal{C}^{r_{0}}}\le1\biggr\}.
\end{align*}
Next, let us define the relevant H\"older space with parabolic scaling. Here, let $B(0,1/2)=\{(t,x)\in\mathbf{R}\times\mathbf{R}^{2}:|t|+|x|<1/2\}$. For $\eta:\mathbf{R}\times \mathbf{R}^2\to\mathbf{R}$ with support contained in $B(0,1/2)$, define
\[
 \mathcal{S}^\lambda_{\mathfrak{s}}\eta(x)=\lambda^{-4}\eta(\lambda^{-2}t,\lambda^{-1}x).
\]
As before, $\mathcal{S}^\lambda_{\mathfrak{s}}\eta$ can be interpreted as a function on $\mathbf{R} \times \mathbf{T}^2$ by spatial periodization. 
Then define $\mathcal{C}_{\mathfrak{s}}^{\alpha}(\mathbf{R}\times\mathbf{T}^{2})$
to be the completion of $\mathcal{C}_{\mathrm{c}}^{\infty}(\mathbf{R}\times\mathbf{T}^{2})$
under the norm
\begin{align*}
\|f\|_{\mathcal{C}_{\mathfrak{s}}^{\alpha}(\mathbf{R}\times\mathbf{T}^{2})}&=\sup\biggl\{\lambda^{-\alpha}\int_{-\infty}^{\infty}\int f(t,x)\mathcal{S}^\lambda_\mathfrak{s}\eta(t-s,x-y)\,\dif x\,\dif t: \\
&\qquad \qquad \qquad  0<\lambda<1,\, (s,y)\in\mathbf{R}\times\mathbf{T}^{2}, \\
&\qquad \qquad \qquad   \eta\in\mathcal{C}_{c}^{r_{0}}(B(0,1/2)), \, \|\eta\|_{\mathcal{C}^{r_{0}}}\le1\biggr\}.
\end{align*}
Furthermore, define $\mathcal{C}_{\mathfrak{s};\mathrm{loc}}^{\alpha}(\mathbf{R}_{> 0}\times\mathbf{T}^{2})$
to be the completion of $\mathcal{C}_{\mathrm{c}}^{\infty}(\mathbf{R}\times\mathbf{T}^{2})$
under the family of seminorms indexed by $\chi\in\mathcal{C}_{c}^{\infty}(\mathbf{R}_{> 0}\times\mathbf{T}^{2})$
(in particular, supported on a compact set that does not intersect
$\{t=0\}$) given by 
\[
f\mapsto\|\chi f\|_{\mathcal{C}_{\mathfrak{s}}^{\alpha}(\mathbf{R}\times\mathbf{T}^{2})}. 
\]
Now we quote the key result of \cite{FM17}, specialized in two different ways to our setting. In the following, a random distribution $f$ is called `translation-invariant' if $f(\cdot +x_0)$ has the  same law as $f(\cdot)$ for any fixed $x_0\in \mathbf{T}^2$. 
\begin{thm}[Fixed-time version; \cite{FM17}]
\label{thm:tightness-fixedtime} Suppose that $p\in[1,\infty)$, $r\in\mathbb{N}$, and
\[
-r<\alpha<\beta-\frac{2}{p}<\beta<0.
\]
Then there exists a function $\phi\in\mathcal{C}_c^r((-1/2,1/2)^2)$ and a finite set
$\Psi\subset\mathcal{C}^r_{c}((-1/2,1/2)^{2})$  so that the
following holds. Let $\{f_{m}\}_{m\ge 1}$ be a 
family of translation-invariant random elements of $\mathcal{C}^r_{c}(\mathbf{T}^{2})^{*}$ such that for some constant $C<\infty$, we have 
\[
\sup_{m\ge 1}\mathbf{E}\left|\int f_{m}(y)\phi(y)\,\dif y\right|^{p}\le C,
\]
and, for all $n\ge1$,
\[
\sup_{m\ge 1}\sup_{\psi\in\Psi}\mathbf{E}\left|\int f_{m}(y)\mathcal{S}^{2^{-n}}\psi(y)\,\dif y\right|^{p}\le C\cdot2^{-np\beta}.
\]
Then $\{f_{m}\}_{m\ge 1}$ is tight in $\mathcal{C}^{\alpha}(\mathbf{T}^{2})$.
\end{thm}
\begin{thm}[In the parabolic scaling; \cite{FM17,G18}]\label{thm:tightness-parabolic}
Suppose that $p\in[1,\infty)$, $r\in\mathbb{N}$, and 
\[
-r<\alpha<\beta-\frac{4}{p}<\beta<0.
\]
Then there is a function $\phi\in\mathcal{C}^r_{c}((0,1)\times(-1/2,1/2)^{2})$
and a finite set $\Psi\subset\mathcal{C}^r_{c}((0,1)\times(-1/2,1/2)^{2})$
so that the following holds. 
Let $\{f_{m}\}_{m\ge 1}$ be a 
family of space-translation-invariant random elements of $\mathcal{C}^r_{c}(\mathbf{R}_{>0}\times\mathbf{T}^{2})^{*}$  so that,
for each $k\ge 1$, there is a constant $C(k)<\infty$ such
that
\[
\sup_{m\ge 1}\sup_{t\in [2^{-2k+1},k]}\mathbf{E}\left|\int_{0}^{\infty}\int f_{m}(s,y)\mathcal{S}^{2^{-k}}_{\mathfrak{s}}\phi(t-s,y)\,\dif y\,\dif s\right|^{p}\le C(k),
\]
and, for all $n\ge k$,
\begin{align*}
&\sup_{m\ge 1}\sup_{\psi\in\Psi}\sup_{{t\in [2^{-2k+1},k]}}
\mathbf{E}\left|\int_{0}^{\infty}\int f_{m}(s,y)\mathcal{S}^{2^{-n}}_{\mathfrak{s}}\psi(t-s,y)\,\dif y\,\dif s\right|^{p} \\
&\quad \qquad\le C(k)\cdot2^{-np\beta}.
\end{align*}
Then $\{f_{m}\}_{m\ge 1}$ is tight in $\mathcal{C}_{\mathfrak{s};\mathrm{loc}}^{\alpha}(\mathbf{R}_{>0}\times\mathbf{T}^{2})$.
\end{thm}

There are two differences between Theorem 2.30 of \cite{FM17} and \cref{thm:tightness-fixedtime,thm:tightness-parabolic} 
as we have stated them. The first is that \cite[Theorem~2.30]{FM17} is stated for the case of subsets of $\mathbf{R}^d$ rather than for the torus. This is no obstacle at all, because we can identify functions on $\mathbf{T}^2$ with $\mathbf{Z}^2$-periodic functions on $\mathbf{R}^2$, and it is easy to check that convergence in $\mathcal{C}^\alpha_\mathrm{loc}$ of a periodic sequence of distributions on $\mathbf{R}^2$ is the same as convergence in $\mathcal{C}^\alpha$ of the corresponding sequence of distributions on $\mathbf{T}^2$.

The second difference is that \cite[Theorem~2.30]{FM17} is stated for the case where all of the coordinates of $\mathbf{R}^d$ are scaled uniformly. It is more natural in our space-time setting to use the parabolic scaling $\mathfrak{s}$, since this scaling leaves the Laplacian invariant. As previously observed in \cite[proof of Theorem~3.10 on p.~26]{G18}, going through the proof of \cite[Theorem~2.30]{FM17}, but using the scaling framework described in \cite[Sections~3 and~10]{H14}, yields \cref{thm:tightness-parabolic}. (In our case, this means scaling time by twice the scaling of space in all places in the argument.)

We note that \cref{thm:tightness-parabolic}, in the language of \cite{FM17}, corresponds to choosing the ``spanning set''  $\{(K_k,k)\}$ with $K_k = [2^{-2k+1},k]\times\mathbf{T}^2$. Of course, the upper bound $k$  is quite arbitary: any function $f(k)$ satisfying $\lim_{k\to\infty}f(k)=\infty$ would do. The lower bound, of course, is required to be greater than $2^{-2k}$ so that the functions $f_m$ are only integrated over positive values. 

\begin{rem}
\label{rem:prodstruct}As pointed out in the discussion following
\cite[Theorem 2.7]{FM17}, the functions $\phi$ and $\psi\in\Psi$
in \cref{thm:tightness-fixedtime,thm:tightness-parabolic} 
  can be taken to be products of univariate functions of
each coordinate. (They are the wavelets of \cite{Dau88}.) We will use 
the product structure of the wavelets to simplify the proof of \cref{thm:maintheorem-quantitative}, stated below.
\end{rem}

Convergence in local negative H\"older spaces means convergence when integrated against a test function, \emph{locally uniformly} in the choice of sufficiently smooth test function up to the rate of blowup as the test functions are scaled. In particular, the topology of a negative H\"older space is stronger than the topology of $\mathcal{D}'$, so convergence in a negative H\"older space implies convergence in $\mathcal{D}'$. Moreover, the spaces $\mathcal{C}^\alpha(\mathbf{T}^2)$ and $\mathcal{C}^\alpha_{\mathfrak{s};\mathrm{loc}}(\mathbf{R}_{>0}\times\mathbf{T}^2)$ are both Polish spaces \cite[Remarks~2.4 and~2.20]{FM17}. Thus, in light of Prokhorov's theorem, the following theorem is a more quantitative version of  \cref{thm:maintheorem}. 

\begin{thm}
\label{thm:maintheorem-quantitative}There is a $\theta_{0}>0$ such 
that if $\theta\in[0,\theta_{0}]$, then for any $\delta > 0$, the family $(h^\eps_\theta)_{\eps>0}$ is a tight family of random distributions in $\mathcal{C}_{\mathfrak{s};\mathrm{loc}}^{-2-\delta}(\mathbf{R}_{>0}\times\mathbf{T}^{2})$, and for any $t>0$, the family $(h^\eps_\theta(t,\cdot))_{\eps>0}$ is  a tight family of random distributions in $\mathcal{C}^{-1-\delta}(\mathbf{T}^{2})$.
\end{thm}

The limited regularities $-2-\delta$ and $-1-\delta$ in the statement
of \cref{thm:maintheorem-quantitative} arise because we are only able
to control the $p=2$ case of the bounds required by \cref{thm:tightness-fixedtime,thm:tightness-parabolic}. 
We expect that higher moments should be bounded similarly, and thus
we make the following conjecture.
\begin{conjecture}
\label{conj:betterregularity}
For any $\delta>0$, the tightness statements in Theorem~\ref{thm:maintheorem-quantitative} hold in the spaces  $\mathcal{C}_{\mathfrak{s};\mathrm{loc}}^{-\delta}(\mathbf{R}_{>0}\times\mathbf{T}^{2})$
and $\mathcal{C}^{-\delta}(\mathbf{T}^{2})$, respectively.
\end{conjecture}

\section{Malliavin calculus}\label{sec:gaussian}
We will use several elementary aspects of the Malliavin calculus in the proofs of our theorems. In this section we recall only the facts that we will use. We refer the reader to Chapter~1 of \cite{Nua06} for an introduction to the Malliavin calculus.

For a random variable $Y$ of the form
\[
 Y = f\left(\int_0^t \int \mathbf{g}(u,x)W(\dif x\,\dif u)\right)
\]
with $f:\mathbf{R}^J\to\mathbf{R}$ smooth and $\mathbf{g}:[0,t]\times\mathbf{T}^2\to\mathbf{R}^J$, we recall that the Malliavin derivative of $Y$ is given by, for $s\in[0,t]$ and $y\in\mathbf{T}^2$,
\begin{equation}
 \Dif_{s,y}Y =\mathbf{g}(s,y)\cdot \nabla f\left(\int_0^t \int \mathbf{g}(u,x)W(\dif x\,\dif u)\right).\label{eq:malldef}
\end{equation}
(Of course, the Malliavin derivative can be defined for more general random variables, but for simplicity we specialize to the case we will use.) The Malliavin derivative satisfies the chain rule
\[
 \Dif_{s,y} h(Y) = h'(Y)\Dif_{s,y}Y,
\]
and the product rule
\[
 \Dif_{s,y} (YZ) = Y\Dif_{s,y}Z +  Z\Dif_{s,y}Y.
\]
We will use two key facts about Malliavin derivatives in our computations, which we state in the following two propositions. In each statement $Y$ is as above.

\begin{prop}[Gaussian integration by parts]
 We have
\begin{equation}
\mathbf{E} \left(Y \int_0^t \int \xi(s,y)W(\dif y\,\dif s)\right) = \mathbf{E}\left(\int_0^t \int \xi(s,y)\Dif_{s,y}Y\,\dif y\,\dif s\right).\label{eq:GIP}
\end{equation}
\end{prop}

For the proof, see \cite[Lemma 1.2.1]{Nua06}.

\begin{prop}[Gaussian Poincar\'e inequality]
We have
\begin{equation}
 \Var Y\le\int_0^t\int \mathbf{E}(\Dif_{s,y} Y)^2\,\dif y\,\dif s.\label{eq:GPI}
\end{equation}
\end{prop}
This was proved in \cite{HPA95}; the statement in our setting was given in  \cite{NPR09}.

\section{Preliminary computations\label{sec:derivatives}}
Having introduced all necessary notation and results from the literature, we are now ready to begin our proofs. Throughout the rest of the paper, we will use $C$ to denote arbitrary universal constants, whose values may change from line to line. Sometimes $C'$, $C_1$ and $C_2$ will be used for the same purpose. As mentioned before, any integral without a specified domain of integration will denote integration over $\mathbf{T}^2$. Unless otherwise mentioned, $L^p$ norms will stand for $L^p$ norms over $\mathbf{T}^2$. We will assume that $\eps<1$ throughout, and sometimes even smaller. 
We will also frequently interchange expectations and integrals, and will move Malliavin derivatives inside integrals and expectations. Since the integrations take place over finite measure spaces and the functions under consideration are smooth in the variables that are differentiated, these manipulations are easily justified.
\subsection{Derivatives}

In this section we derive compute several quantities that appear in expressions for the moments
appear in the hypotheses of \cref{thm:tightness-fixedtime,thm:tightness-parabolic}, as well as in the derivatives of these moments. In \cref{sec:Taylorbound}, we will show how to use the derivatives to control Taylor-like expansions of the moments, while in \cref{sec:maintheoremproof}, we will show how to use these bounds to estimate the moments appearing in  \cref{thm:tightness-fixedtime,thm:tightness-parabolic} and thus prove \cref{thm:maintheorem-quantitative}. The reader may at this point wish to flip forward to \cref{lem:varcomp} to see how the expressions in this subsection appear in the variance bound.

In
order to
write our statements, we first need to introduce some notation. We
define the tilted probability measure $\widehat{\mathbb{P}}_{X^{t,x}}^{\eps, W, \theta}$
according to the Radon--Nikodym derivative
\begin{equation}
\frac{\dif\widehat{\mathbb{P}}_{X^{t,x}}^{\eps,W,\theta}}{\dif\mathbb{P}_{X^{t,x}}}=\frac{\mathscr{E}_{\theta,t}^{\eps}[W,X]}{\mathbb{E}_{X^{t,x}}\mathscr{E}_{\theta,t}^{\eps}[W,X]}\label{eq:RNderiv-1}
\end{equation}
and let $\widehat{\mathbb{E}}_{X^{t,x}}^{\eps,W,\theta}$ denote expectation
with respect to this measure. (Here, $\mathbb{P}_{X^{t,x}}$ is the
measure corresponding to $\mathbb{E}_{X^{t,x}}$ defined in \cref{sec:CH}.)
Note that the Radon--Nikodym derivative \cref{eq:RNderiv-1} is random,
as it depends on the noise.

Our derivative computations will involve functions of multiple Brownian
paths. If $\mathbf{x}=(x_{1},\ldots,x_{J})$ and $\mathbf{X}=(X_{1},\ldots,X_{J})$,
we will frequently use the shorthand notations
\[
\mathbb{P}_{\mathbf{X}^{t,\mathbf{x}}}=\mathbb{P}_{X_{1}^{t,x_{1}},\ldots,X_{J}^{t,x_{J}}}:=\mathbb{P}_{X_{1}^{t,x_{1}}}\otimes\cdots\otimes\mathbb{P}_{X_{J}^{t,x_{J}}}
\]
and
\[
\widehat{\mathbb{P}}_{\mathbf{X}^{t,\mathbf{x}}}^{\theta,W,\eps}=\widehat{\mathbb{P}}_{X_{1}^{t,x_{1}},\ldots,X_{J}^{t,x_{J}}}^{\theta,W,\eps}:=\widehat{\mathbb{P}}_{X_{1}^{t,x_{1}}}^{\theta,W,\eps}\otimes\cdots\otimes\widehat{\mathbb{P}}_{X_{J}^{t,x_{J}}}^{\theta,W,\eps}.
\]
We also define
\[
\mathscrbf{E}_{\theta,t}^{\eps}[W,\mathbf{X}]=\prod_{j=1}^{J}\mathscr{E}_{\theta,t}^{\eps}[W,X_{j}],
\]
so that
\[
\frac{\dif\widehat{\mathbb{P}}_{\mathbf{X}^{t,\mathbf{x}}}^{\eps,W,\theta}}{\dif\mathbb{P}_{\mathbf{X}^{t,\mathbf{x}}}}=\frac{\mathscrbf{E}_{\theta,t}^{\eps}[W,\mathbf{X}]}{\mathbb{E}_{\mathbf{X}^{t,\mathbf{x}}}\mathscrbf{E}_{\theta,t}^{\eps}[W,\mathbf{X}]}.
\]
The product measures defined above will often be used in the following way. Suppose that we want to evaluate $(\widehat{\mathbb{E}}_{X^{t,x}}^{\eps,W,\theta}\mathscr{Q}(X))^2$ for some functional $\mathscr{Q}$. Then we will use the representation
\[
(\widehat{\mathbb{E}}_{X^{t,x}}^{\eps,W,\theta}\mathscr{Q}(X))^2 = \widehat{\mathbb{E}}_{X_1^{t,x}, X_2^{t,x}}^{\eps,W,\theta}(\mathscr{Q}(X_1)\mathscr{Q}(X_2)),
\]
which conveniently allows exchange of expectations and integrals in many places, which we would not be able to achieve with the expectation squared.

Next, the intersection time of two paths $X_{1},X_{2}$ is defined as 
\[
\mathscr{I}_{t}^{\eps}[X_{1},X_{2}]=\int_{0}^{t}R^{\eps}(X_{1}(s)-X_{2}(s))\,\dif s,
\]
where $R^\eps$ is defined by the $\mathbf{T}^2$-convolution
\[
R^{\eps}=\rho^{\eps}*\rho^{\eps}.
\]
Since $\rho$ is an even function, note that
\begin{align*}
R^\eps(0) &= \int \rho^\eps(x)^2 \, \dif x\\
&= \int_{\mathbf{R}^2}\eps^{-4}\rho(\eps^{-1}x)^2 \, \dif x = \int_{\mathbf{R}^2} \eps^{-2} \rho(y)^2 \, \dif y. 
\end{align*}
Thus, for any path $X$,
\begin{equation}\label{eq:itsame}
\mathscr{I}_t^\eps[X, X] = \frac{t}{\eps^2} \|\rho\|_{L^2}^2. 
\end{equation}
We will have an important use for the above identity later. 
We now proceed with our derivative computations. We first note that
\begin{align}
\frac{\partial}{\partial\theta}\mathscr{E}_{\theta,t}^{\eps}[W,X]&=\biggl(\frac{1}{2}(\theta|\log\eps|)^{-\frac{1}{2}}\int_{0}^{t}\int \rho^{\eps}(X(s)-y)W(\dif y\,\dif s)\notag\\
&\qquad \qquad -|\log\eps|^{-\frac{1}{2}}\frac{\partial\kappa_{\theta}^{\eps}(t)}{\partial\theta}\biggr)\mathscr{E}_{\theta,t}^{\eps}[W,X].\label{eq:dbetaFunnyE}
\end{align}
We can also compute the Malliavin derivative $\Dif_{s,y}$  with respect
to the white noise. A simple calculation using \cref{eq:malldef} gives
\begin{equation}
\Dif_{s,y}\mathscr{E}_{\theta,t}^{\eps}[W,X]=\theta^{\frac{1}{2}}|\log\eps|^{-\frac{1}{2}}\rho^{\eps}(X(s)-y)\mathscr{E}_{\theta,t}^{\eps}[W,X].\label{eq:detaE}
\end{equation}
The following lemma is a more involved derivative computation. Let
$\mathbf{X}=(X_{1},\ldots,X_{J})$, $\widetilde{\mathbf{X}}=(\widetilde{X}_{1},\ldots,\widetilde{X}_{J})$,
and $\mathbf{x}=(x_{1},\ldots,x_{J})$. Suppose that $\mathscr{Q}:\mathcal{C}([0,t])^{J}\to\mathbf{R}$
is measurable.
\begin{lem}
\label{lem:Pderivs}We have the derivative formula 
\begin{align}
&\frac{\partial}{\partial\theta}\widehat{\mathbb{E}}_{\mathbf{X}^{t,\mathbf{x}}}^{\theta,W,\eps}\mathscr{Q}[\mathbf{X}]\notag\\
&=\frac{1}{2}(\theta|\log\eps|)^{-\frac{1}{2}}\int_{0}^{t}\int\sum_{k=1}^{J}\widehat{\mathbb{E}}_{\mathbf{X}^{t,\mathbf{x}},\widetilde{\mathbf{X}}^{t,\mathbf{x}}}^{\theta,W,\eps}\biggl(\mathscr{Q}[\mathbf{X}](\rho^{\eps}(X_{k}(s)-y)\notag \\
&\qquad \qquad \qquad\qquad \qquad -\rho^{\eps}(\widetilde{X}_{k}(s)-y))\biggr)W(\dif y\,\dif s).\label{eq:ddtheta}
\end{align}
Moreover, whenever $s\in[0,t]$ we have
\begin{align}
&\Dif_{s,y}\frac{\mathscrbf{E}_{\theta,t}^{\eps}[W,\mathbf{X}]}{\mathbb{E}_{\mathbf{X}^{t,\mathbf{x}}}\mathscrbf{E}_{\theta,t}^{\eps}[W,\mathbf{X}]}\notag\\
&=\theta^{\frac{1}{2}}|\log\eps|^{-\frac{1}{2}}\sum_{k=1}^{J}\biggl[\frac{\rho^{\eps}(X_{k}(s)-y)\mathscrbf{E}_{\theta,t}^{\eps}[W,\mathbf{X}]}{\mathbb{E}_{\mathbf{X}^{t,\mathbf{x}}}\mathscrbf{E}_{\theta,t}^{\eps}[W,\mathbf{X}]}\notag\\
&\qquad \qquad-\frac{\mathscrbf{E}_{\theta,t}^{\eps}[W,\mathbf{X}]\mathbb{E}_{\mathbf{X}^{t,\mathbf{x}}}(\rho^{\eps}(X_{k}(s)-y)\mathscrbf{E}_{\theta,t}^{\eps}[W,\mathbf{X}])}{(\mathbb{E}_{\mathbf{X}^{t,\mathbf{x}}}\mathscrbf{E}_{\theta,t}^{\eps}[W,\mathbf{X}])^{2}}\biggr],\label{eq:malliavinderiv-noE}
\end{align}
and
\begin{align}
&\Dif_{s,y}\widehat{\mathbb{E}}_{\mathbf{X}^{t,\mathbf{x}}}^{\theta,W,\eps}\mathscr{Q}[\mathbf{X}]\notag\\
&=\sum_{k=1}^{J}\theta^{\frac{1}{2}}|\log\eps|^{-\frac{1}{2}}\widehat{\mathbb{E}}_{\mathbf{X}^{t,\mathbf{x}},\widetilde{\mathbf{X}}^{t,\mathbf{x}}}^{\theta,W,\eps}\mathscr{Q}[\mathbf{X}][\rho^{\eps}(X_{k}(s)-y)-\rho^{\eps}(\widetilde{X}_{k}(s)-y)].\label{eq:malliavinderiv}
\end{align}
\end{lem}

\begin{lem}
\label{lem:fullexpexpansion}We have
\begin{align}
\frac{\partial}{\partial\theta} & \mathbf{E}\widehat{\mathbb{E}}_{\mathbf{X}^{t,\mathbf{x}}}^{\theta,W,\eps}\mathscr{Q}[\mathbf{X}]\nonumber \\
 & =\frac{1}{2}|\log\eps|^{-1}\mathbf{E}\widehat{\mathbb{E}}_{\mathbf{X}^{t,\mathbf{x}},\widetilde{\mathbf{X}}^{t,\mathbf{x}},\widetilde{\widetilde{\mathbf{X}}}^{t,\mathbf{x}}}^{\theta,W,\eps}\biggl[\mathscr{Q}[\mathbf{X}]\sum_{\substack{k,\ell=1}
}^{J}\biggl(\mathscr{I}_{t}^{\eps}[X_{k},X_{\ell}]\mathbf{1}_{k\ne \ell}\notag\\
&\qquad -2\mathscr{I}_{t}^{\eps}[X_{k},\widetilde{X}_{\ell}]+(1+\mathbf{1}_{k=\ell})\mathscr{I}_{t}^{\eps}[\widetilde{X}_{k},\widetilde{\widetilde{X}}_{\ell}]\biggr)\biggr].\label{eq:fullexpexpansioneq}
\end{align}
\end{lem}
We defer the proofs of \cref{lem:Pderivs} and \cref{lem:fullexpexpansion}
to \cref{subsec:derivcomps}.

\subsection{Brownian motion intersection estimates\label{sec:BMintersectionests}}

In this section we state the results about Brownian motion that
we will need to prove our theorems. The Brownian motion estimates are
quite standard, so we defer the proofs to \cref{subsec:BMcomps}.
The underlying probabilistic facts behind the following lemmas are that
a Brownian motion started at the origin in $\mathbf{R}^{2}$ and run
for time $t\gg1$ will spend time on the order $\log t$ in a
unit ball around the origin, and that a random walk started at distance
$\sqrt{t}$ from the origin in $\mathbf{R}^{2}$ and run for $t$
steps will reach the unit ball around the origin with probability
on the order of $1/\log t$, but conditional on that event
will again spend on the order of $\log t$ steps in the unit ball. 

We will use the notation $|x|_{\mathbf{T}^{2}}$ to mean the distance
in the torus of $x$ from the origin, that is, 
\[
|x|_{\mathbf{T}^{2}}=\min_{y\in\mathbf{Z}^{2}}|x+y|,
\]
where $|x+y|$ is the usual Euclidean norm of $x+y$.

\begin{lem}
\label{lem:logupperbound}There is an absolute constant $C$ so that if $Y\in C([0,t],\mathbf{T}^{2})$
is a deterministic path and $\eps\le\e^{-t/2}$ then we have
\[
\mathbb{E}_{X^{t,x}}\mathscr{I}_{t}^{\eps}[X,Y]^{r}\le C^{r}r!|\log\eps|^{r}.
\]
\end{lem}

\begin{lem}
\label{lem:X0Yest}There is an absolute constant $C$ so that as long as $\eps\le\e^{-t/2}$, 
we have
\begin{equation}
\mathbb{E}_{X_{1}^{t,x_{1}},X_{2}^{t,x_{2}}}\mathscr{I}_{t}^{\eps}[X_{1},X_{2}]^{r}
\le C^{r}r!(t+1+ \log|x_{1}-x_{2}|_{\mathbf{T}^{2}}^{-2})|\log\eps|^{r-1}.\label{eq:IX0Ymoments}
\end{equation}
\end{lem}

\subsection{The renormalization constant}

The following lemma allows us to give a somewhat more explicit
expression for the renormalization constant $\kappa_{\theta}^{\eps}(t)$
defined in \cref{eq:renormalizationconstantexpdef}.
\begin{lem}
\label{lem:expcomp}We have
\begin{equation}
\kappa_{\theta}^{\eps}(t)=\frac{1}{2}|\log\eps|^{-\frac{1}{2}}\left(\frac{\theta t}{\eps^{2}}\|\rho\|_{L^{2}}^{2}-\int_{0}^{\theta}\mathbf{E}\widehat{\mathbb{E}}_{X^{t,x},\widetilde{X}^{t,x}}^{\zeta,W,\eps}\mathscr{I}_{t}^{\eps}[X,\widetilde{X}]\,\dif\zeta\right).\label{eq:renormdef}
\end{equation}
\end{lem}

\begin{proof}
We can compute
\[
\frac{\partial}{\partial\theta}\widetilde{h}_{\theta}^{\eps}(t,x)=\frac{1}{2\sqrt{\theta}}\frac{\mathbb{E}_{X^{t,x}}\left[\left(\int_{0}^{t}\int\rho^{\eps}(X(s)-y)W(\dif y\,\dif s)\right)\mathscr{E}_{\theta,t}^{\eps}[W,X]\right]}{\mathbb{E}_{X^{t,x}}\mathscr{E}_{\theta,t}^{\eps}[W,X]}.
\]
So, using the Malliavin integration by parts formula \cref{eq:GIP}, we have
\begin{align*}
&\frac{\partial}{\partial\theta}\mathbf{E}\widetilde{h}_{\theta}^{\eps}(t,x) \\
&=\frac{1}{2\sqrt{\theta}}\mathbf{E}\int_{0}^{t}\int \mathbb{E}_{X^{t,x}}\biggl[\rho^{\eps}(X(s)-y)\Dif_{s,y}\biggl(\frac{\mathscr{E}_{\theta,t}^{\eps}[W,X]}{\mathbb{E}_{X^{t,x}}\mathscr{E}_{\theta,t}^{\eps}[W,X]}\biggr)\biggr]\,\dif y\,\dif s.
\end{align*}
Thus, by \cref{eq:malliavinderiv-noE}, we have
\begin{align*}
&\frac{\partial}{\partial\theta}\mathbf{E}\widetilde{h}_{\theta}^{\eps}(t,x) \\
&=\frac{1}{2}|\log \eps|^{-\frac{1}{2}} \mathbf{E}\int_0^t\int\widehat{\mathbb{E}}_{X^{t,x},\widetilde{X}^{t,x}}^{\theta,W,\eps}(\rho^\eps(X(s)-y)^2 \\
&\qquad \qquad - \rho_\eps(X(s)-y)\rho_\eps(\widetilde{X}(s)-y))\, \dif y\, \dif s.
\end{align*}
Since $\rho$ is an even function, we have that for any two paths $X$ and $\widetilde{X}$, 
\begin{align}
&\int_0^t \int \rho^\eps(X(s)-y) \rho^\eps(\widetilde{X}(s)-y) \, \dif y\, \dif s\notag\\
&= \int_0^t \int \rho^\eps(X(s)-\widetilde{X}(s)+z) \rho^\eps(z) \, \dif z\, \dif s\notag\\
&= \int_0^t R^\eps(X(s)-\widetilde{X}(s)) \, \dif s = \mathscr{I}_t^\eps[X, \widetilde{X}]. \label{eq:itform}
\end{align}
Therefore, by \cref{eq:itsame}, we get
\begin{align*}
\frac{\partial}{\partial\theta}\mathbf{E}\widetilde{h}_{\theta}^{\eps}(t,x) 
 & =\frac{1}{2}|\log\eps|^{-\frac{1}{2}}\left[\frac{t}{\eps^{2}}\|\rho\|_{L^{2}}^{2}-\mathbf{E}\widehat{\mathbb{E}}_{X^{t,x},\widetilde{X}^{t,x}}^{\theta,W,\eps}\mathscr{I}_{t}^{\eps}[X,\widetilde{X}]\right].
\end{align*}
Integrating in $\theta$, and observing that $\widetilde{h}^\eps_0\equiv0$ by \cref{eq:fkform}, we get \cref{eq:renormdef}.
\end{proof}

\section{Taylor expansion bound\label{sec:Taylorbound}}

Our main technique in this paper is to expand random variables of the form
\[
\mathbf{E}\widehat{\mathbb{E}}_{\mathbf{X}^{t,\mathbf{x}}}^{\theta,W,\eps}\mathscr{Q}[\mathbf{X}],
\]
 in terms
of an infinite Taylor series in $\theta$. (Again, we invite the reader to look ahead to \cref{lem:varcomp} to see how such functionals arise in our variance bound.) An obstacle to carrying out the Taylor expansion is that the right-hand side of \cref{eq:fullexpexpansioneq}
has on the order of $J^{2}$ terms, where as above $J$ is the number
of Brownian motions participating in $\mathscr{Q}$. When the differentiation
process is iterated $r$ times, we see that $J$ grows linearly in
$r$, so the number of terms in the $r$th derivative will be on the
order of $(r!)^{2}$. On its face, this is too many terms for
the Taylor series to be controlled. Of course, the terms will have
different signs and there will be cancellations. However, we do not
know a way to control the cancellation directly.

Instead, a key step in our approach is the following \cref{prop:frderivbound-1},
which uses Young's inequality to ``collapse'' similar terms to \emph{upper-bound}
the $\theta$-derivative of $G_{r,\mathscr{Q}}$, the expectation
of a functional, in terms of $G_{r+1,\mathscr{Q}}$, the expectation
of another functional of a similar form. The key point is that the
successive functionals arising in this way use a \emph{constant} number
of Brownian motions, rather than the linearly-growing number that
arises from na\"ive iterated differentiation as described in the previous
paragraph. This means that the number of terms in the functionals
grows only exponentially rather than like the square of the factorial,
and this can be controlled by taking $\theta$ sufficiently small
in the Taylor series.

The price we pay, of course, is that the functions $G_{r,\mathscr{Q}}$
are not truly successive derivatives of a function; rather, they are
successive upper bounds on each other's derivatives. As we show in
\cref{lem:grtaylor} and \cref{cor:ginfseriesbound} below, this still
allows us to use a Taylor series-like construction to get upper bounds
on our original quantity of interest. However, the fact that our technique
as it stands does not allow us to obtain lower bounds seems to be
a key bottleneck impeding further progress. Obtaining more precise
control on these Taylor series may be a fruitful target of future
work.
\begin{prop}
\label{prop:frderivbound-1}Let $\mathbf{X}_{0}=(X_{0;1},\ldots,X_{0;J})$,
$\mathbf{X}_{1}=(X_{1;1},\ldots,X_{1;J})$, $\mathbf{X}_{2}=(X_{2;1},\ldots,X_{2;J})$ and $\mathbf{x} = (x_1,\ldots,x_J)$.
Define 
\begin{align*}
\underline{\mathbf{X}} & =(\mathbf{X}_{0},\mathbf{X}_{1},\mathbf{X}_{2}),\\
\underline{\mathbf{x}} & =(\mathbf{x},\mathbf{x},\mathbf{x}).
\end{align*}
Suppose that 
\[
\mathscr{Q}:\mathcal{C}([0,t])^{J}\to\mathbf{R}
\]
 is a measurable functional. Define for $r\ge0$
\begin{equation}
F_{r,\mathscr{Q}}(\theta)=\sum_{\alpha=0}^{2}\sum_{\beta=0}^{2}\sum_{j=1}^{J}\sum_{k=1}^{J}\mathbf{1}_{(\alpha,j)\ne(\beta,k)}F_{\alpha,j;\beta,k;r, \mathscr{Q}}(\theta),\label{eq:fdev-1}
\end{equation}
where
\begin{equation}
F_{\alpha,j;\beta,k;r,\mathscr{Q}}(\theta)=\mathbf{E}\widehat{\mathbb{E}}_{\underline{\mathbf{X}}^{t,\underline{\mathbf{x}}}}^{\theta,W,\eps}(\mathscr{Q}[\mathbf{X}_{0}]\mathscr{I}_{t}^{\eps}[X_{\alpha;j},X_{\beta;k}]^{r}).\label{eq:fajbkrdef-1}
\end{equation}
Finally, let 
\begin{equation}
G_{r,\mathscr{Q}}(\theta)=|\log\eps|^{-r}F_{r,\mathscr{Q}}(\theta).\label{eq:gdef}
\end{equation}
Then there is a constant $C(J)$, depending only on $J$, so that
\begin{equation}
|G_{r,\mathscr{Q}}'(\theta)|\le C(J)G_{r+1, |\mathscr{Q}|}(\theta).\label{eq:gderivbound}
\end{equation}
 (Here, $|\mathscr{Q}|$ is defined by $|\mathscr{Q}|[\mathbf{X}_0]=|\mathscr{Q}[\mathbf{X}_0]|$.)
\end{prop}

\begin{proof}
Let us define 
\begin{align*}
\underline{\mathbf{X}}_{*} & =(\underline{\mathbf{X}}^{t,\underline{\mathbf{x}}},\widetilde{\underline{\mathbf{X}}}^{t,\underline{\mathbf{x}}},\widetilde{\underline{\widetilde{\mathbf{X}}}}^{t,\underline{\mathbf{x}}}),\\
\underline{\mathbf{x}}_{*} & =(\underline{\mathbf{x}},\underline{\mathbf{x}},\underline{\mathbf{x}}).
\end{align*}
Let $(\alpha,j)\ne(\beta,k)$. By \cref{lem:fullexpexpansion}, we have
\begin{align}
2|\log\eps|F_{\alpha,j;\beta,k;r,\mathscr{Q}}'(\theta)=\sum_{\alpha'=0}^{2}\sum_{\beta'=0}^{2}\sum_{j'=1}^{J}\sum_{k'=1}^{J}\mathbf{E}\widehat{\mathbb{E}}_{\underline{\mathbf{X}}_{*}^{t,\underline{\mathbf{x}}_{*}}}^{\theta,W,\eps}\mathscr{J}_{\alpha',\beta',j',k'}^{\alpha,\beta,j,k}[\underline{\mathbf{X}}_{*}],\label{eq:fajbkrprime}
\end{align}
where
\begin{align*}
&\mathscr{J}_{\alpha',\beta',j',k'}^{\alpha,\beta,j,k}[\underline{\mathbf{X}}_{*}] \notag\\
& =\mathscr{Q}[\mathbf{X}_{0}]\mathscr{I}_{t}^{\eps}[X_{\alpha;j},X_{\beta;k}]^{r}\biggl(\mathscr{I}_{t}^{\eps}[X_{\alpha';j'},X_{\beta';k'}]\mathbf{1}_{(\alpha',j')\ne(\beta',k')}\notag\\
 & \qquad-2\mathscr{I}_{t}^{\eps}[X_{\alpha';j'},\widetilde{X}_{\beta';k'}]+(1+\mathbf{1}_{(\alpha',j')=(\beta',k')})\mathscr{I}_{t}^{\eps}[\widetilde{X}_{\alpha';j'},\widetilde{\widetilde{X}}_{\beta';k'}]\biggr).
\end{align*}
By the triangle inequality and Young's inequality, we have
\begin{align}
|\mathscr{J}_{\alpha',\beta',j',k'}^{\alpha,\beta,j,k}[\underline{\mathbf{X}}_{*}]| & \le\frac{5r}{r+1}|\mathscr{Q}[\mathbf{X}_{0}]|\mathscr{I}_{t}^{\eps}[X_{\alpha;j},X_{\beta;k}]^{r+1}\notag\\
&\quad +\frac{1}{r+1}|\mathscr{Q}[\mathbf{X}_{0}]|\mathscr{I}_{t}^{\eps}[X_{\alpha';j'},X_{\beta';k'}]^{r+1}\mathbf{1}_{(\alpha',j')\ne(\beta',k')}\nonumber \\
 & \quad+\frac{2}{r+1}|\mathscr{Q}[\mathbf{X}_{0}]|\mathscr{I}_{t}^{\eps}[X_{\alpha';j'},\widetilde{X}_{\beta';k'}]^{r+1}\notag\\
 &\quad +\frac{2}{r+1}|\mathscr{Q}[\mathbf{X}_{0}]|\mathscr{I}_{t}^{\eps}[\widetilde{X}_{\alpha';j'},\widetilde{\widetilde{X}}_{\beta';k'}]^{r+1}.\label{eq:Youngs}
\end{align}
Let $f:\{0,1,2\}\to \{0,1,2\}$ be an arbitrary function that has the property that $f(\alpha)\in \{1,2\}\setminus \{\alpha\}$ for any $\alpha$. By symmetry, whenever $(\alpha', j')\ne (\beta', k')$, 
\begin{align*}
&\mathbf{E}  \widehat{\mathbb{E}}_{\underline{\mathbf{X}}_{*}^{t,\underline{\mathbf{x}}_{*}}}^{\theta,W,\eps}( |\mathscr{Q}[\mathbf{X}_{0}]|\mathscr{I}_{t}^{\eps}[X_{\alpha';j'},\widetilde{X}_{\beta';k'}]^{r+1})\\
&= 
\mathbf{E} \widehat{\mathbb{E}}_{\underline{\mathbf{X}}_{*}^{t,\underline{\mathbf{x}}_{*}}}^{\theta,W,\eps}(|\mathscr{Q}[\mathbf{X}_{0}]|\mathscr{I}_{t}^{\eps}[X_{\alpha';j'},X_{f(\alpha');k'}]^{r+1}),
\end{align*}
and similarly,
\begin{align*}
&\mathbf{E} \widehat{\mathbb{E}}_{\underline{\mathbf{X}}_{*}^{t,\underline{\mathbf{x}}_{*}}}^{\theta,W,\eps}(|\mathscr{Q}[\mathbf{X}_{0}]|\mathscr{I}_{t}^{\eps}[\widetilde{X}_{\alpha';j'},\widetilde{\widetilde{X}}_{\beta';k'}]^{r+1}) \\
&= \mathbf{E} \widehat{\mathbb{E}}_{\underline{\mathbf{X}}_{*}^{t,\underline{\mathbf{x}}_{*}}}^{\theta,W,\eps}(|\mathscr{Q}[\mathbf{X}_{0}]|\mathscr{I}_{t}^{\eps}[X_{1;j'},X_{2;k'}]^{r+1}).
\end{align*}
Combining these observations with \cref{eq:Youngs}, we have
\begin{align*}
\mathbf{E} & \widehat{\mathbb{E}}_{\underline{\mathbf{X}}_{*}^{t,\underline{\mathbf{x}}_{*}}}^{\theta,W,\eps}\mathscr{J}_{\alpha',\beta',j',k'}^{\alpha,\beta,j,k}[\underline{\mathbf{X}}_{*}]\\
 & \le\frac{5r}{r+1}F_{\alpha,j;\beta,k;r+1,|\mathscr{Q}|}(\theta)+\frac{\mathbf{1}_{(\alpha',j')\ne(\beta',k')}}{r+1}F_{\alpha',j';\beta',k';r+1,|\mathscr{Q}|}(\theta)\\
 & \qquad+\frac{2}{r+1}F_{\alpha',j';f(\alpha'),k';r+1,|\mathscr{Q}|}(\theta)+\frac{2}{r+1}F_{1,j';2,k';r+1,|\mathscr{Q}|}(\theta)\\
 & \le5F_{\alpha,j;\beta,k;r+1,|\mathscr{Q}|}(\theta)+\mathbf{1}_{(\alpha',j')\ne(\beta',k')}F_{\alpha',j';\beta',k';r+1,|\mathscr{Q}|}(\theta)\\
 &\qquad +2F_{\alpha',j';f(\alpha'),k';r+1,|\mathscr{Q}|}(\theta)+2F_{1,j';2,k';r+1,|\mathscr{Q}|}(\theta).
\end{align*}
In light of \cref{eq:gdef} and \cref{eq:fajbkrprime}, this implies \cref{eq:gderivbound}
with $C(J)=CJ^{2}$ for some absolute constant $C$.
\end{proof}
\begin{lem}
\label{lem:grtaylor}For any $K\ge0$, we have
\[
G_{r,\mathscr{Q}}(\theta)\le\sum_{j=0}^{K}\frac{(C(J)\theta)^{j}}{j!}G_{r+j, |\mathscr{Q}|}(0)+\frac{(C(J)\theta)^{K+1}}{(K+1)!}\sup_{0\le\theta'\le\theta}G_{r+K+1, |\mathscr{Q}|}(\theta').
\]
\end{lem}

This statement is proved from \cref{eq:gderivbound} in the same way
as Taylor's theorem from single-variable calculus, using inequalities
instead of equalities, so we omit the details.%

\begin{cor}
\label{cor:ginfseriesbound}
For any bounded measurable functional $\mathscr{Q}$ and any $r\ge 0$,
\[
G_{r, \mathscr{Q}}(\theta)\le\sum_{j=0}^{\infty}\frac{(C(J)\theta)^{j}}{j!}G_{r+j, |\mathscr{Q}|}(0).
\]
\end{cor}

\begin{proof}
Let $\|\mathscr{Q}\|_\infty$ be an absolute bound on $|\mathscr{Q}|$. Then we  have a trivial bound
\[
G_{r,\mathscr{Q}}(\theta)\le F_{r,\mathscr{Q}}(\theta)\le9J^{2}\|\mathscr{Q}\|_{\infty}(t\|\rho\|_{L^\infty}/\eps^2)^{r}
\]
by the definition \cref{eq:fdev-1}--\cref{eq:fajbkrdef-1} of $F_{r,\mathscr{Q}}$, and the observation that
\begin{align*}
\|\mathscr{I}_t^\eps\|_{\infty} &\le t\|R^\eps\|_{\mathcal{L}^\infty}\le t \|\rho^\eps\|_{L^1}\|\rho^\eps\|_{L^\infty} = \frac{t\|\rho\|_{L^\infty} }{\eps^2}.
\end{align*}
Combining this with \cref{lem:grtaylor}, we have that
\begin{align*}
&G_{r,\mathscr{Q}}(\theta)  \le\sum_{j=0}^{K}\frac{(C(J)\theta)^{j}}{j!}G_{r+j, |\mathscr{Q}|}(0)+\frac{(C(J)\theta)^{K+1}}{(K+1)!}\sup_{0\le\theta'\le\theta}G_{r+K+1, |\mathscr{Q}|}(\theta')\\
 & \le\sum_{j=0}^{K}\frac{(C(J)\theta)^{j}}{j!}G_{r+j, |\mathscr{Q}|}(0)\\
 &\qquad \qquad +9J^{2}\|\mathscr{Q}\|_{\infty}(t\|\rho\|_{L^\infty}/\eps^2)^{r+K+1}\frac{(C(J)\theta)^{K+1}}{(K+1)!}.
\end{align*}
The result follows when we notice that the remainder term goes to
$0$ as $K$ goes to infinity.
\end{proof}
We conclude this section with two examples of how we can apply the above
bounds in concert with the results of \cref{sec:BMintersectionests}. First, we compute the first-order asymptotics of the renormalization constant.

\begin{lem}
\label{lem:kappaasympt}For fixed $\theta$ sufficiently small and
fixed $t$, we have 
\[
\kappa_{\theta}^{\eps}(t)=\frac{1}{2}|\log\eps|^{-\frac{1}{2}}\frac{\theta t}{\eps^{2}}\|\rho\|_{L^{2}}^{2}+O(|\log\eps|^{\frac{1}{2}})
\]
as $\eps\downarrow0$.
\end{lem}

\begin{proof}Let $\mathbf{x}=(x,x)$ and fix notation as in \cref{prop:frderivbound-1}. We have by \cref{cor:ginfseriesbound} that, as long as $\theta$ is sufficiently small and $\eps<\e^{-t/2}$, 
\begin{align*}
 \mathbf{E}&\widehat{\mathbb{E}}_{X^{t,x}_{0,1},X^{t,x}_{0,2}}^{\theta,W,\eps}\mathscr{I}_{t}^{\eps}[X_{0,1},X_{0,2}]\\
 &\le\sum_{\ell=0}^{\infty}\frac{C^{\ell}\theta^{\ell}}{\ell!|\log\eps|^{\ell}}\biggl(\sum_{\alpha=0}^{2}\sum_{\beta=0}^{2}\sum_{j=1}^{2}\sum_{k=1}^{2}\notag\\
&\qquad \qquad \mathbf{1}_{(\alpha,j)\ne(\beta,k)}\mathbb{E}_{\mathbf{X}_{*}^{t,\mathbf{x}_{*}}}(\mathscr{I}_{t}^{\eps}[X_{0,1},X_{0,2}]\mathscr{I}_{t}^{\eps}[X_{\alpha,j},X_{\beta,k}]^{\ell})\biggr)\\
&\le  35\sum_{\ell=0}^{\infty}\frac{C^{\ell}\theta^{\ell}}{\ell!|\log\eps|^{\ell}}\mathbb{E}_{{X^{t,x}_{0,1},X^{t,x}_{0,2}}}\mathscr{I}_{t}^{\eps}[X_{0,1},X_{0,2}]^{\ell+1}\\
&\le  C'|\log\eps|,
\end{align*}
where the second inequality is by Young's inequality and the third is by \cref{lem:logupperbound}. Then the statement follows
from \cref{eq:renormdef}.
\end{proof}

The following proposition is used in \cref{sec:nontriviality}.
In \cref{sec:maintheoremproof} we use a slightly more specialized,
but similar in flavor, application of the bounds in \cref{sec:BMintersectionests};
see \cref{prop:fajbkrbound}.
\begin{prop}
\label{prop:simpleUB}
Let $\mathbf{X} = (X_1,\ldots, X_J)$ and $\mathbf{x}=(x_1,\ldots,x_J)$. Let $\mathscr{Q}$ be a bounded measurable functional  and define
\[
Q(\theta)=\mathbf{E}\widehat{\mathbb{E}}_{\mathbf{X}^{t,\mathbf{x}}}^{\theta,W,\eps}\mathscr{Q}[\mathbf{X}].
\]
Then there are constants $C>0$ and $\theta_{0}>0$, depending only on $J$ (and not on $\eps$, $\theta$, $\mathscr{Q}$, $\mathbf{x}$ and $t$), so that if $0\le\theta\le\theta_{0}$ and $\eps\le\e^{-t/2}$,
then
\begin{equation}
|Q'(\theta)|\le C(\mathbb{E}_{\mathbf{X}^{t,\mathbf{x}}}|\mathscr{Q}[\mathbf{X}]|^{2})^{1/2}.\label{eq:dAtildebound}
\end{equation}
\end{prop}

\begin{proof}
By \cref{prop:frderivbound-1}, we get
\[
|Q'(\theta)|\le C(J) G_{1, |\mathscr{Q}|}(\theta). 
\]
Therefore by \cref{cor:ginfseriesbound}, with $\mathbf{X}_{*}$
defined as in the statement of \cref{prop:frderivbound-1} (with $\mathbf{X}_0=\mathbf{X}$), we get
\begin{align}
|Q'(\theta)|&\le \sum_{\ell=0}^{\infty}\frac{(C(J)\theta)^{\ell}}{\ell!}G_{\ell+1, |\mathscr{Q}|}(0)\notag \\
&=\sum_{\ell=0}^{\infty}\frac{C(J)^{\ell}\theta^{\ell}}{\ell!|\log\eps|^{\ell+1}}\biggl(\sum_{\alpha=0}^{2}\sum_{\beta=0}^{2}\sum_{j=1}^{2}\sum_{k=1}^{2}\notag\\
&\qquad \qquad \mathbf{1}_{(\alpha,j)\ne(\beta,k)}\mathbb{E}_{\mathbf{X}_{*}^{t,\mathbf{x}_{*}}}(|\mathscr{Q}[\mathbf{X}_{0}]|\mathscr{I}_{t}^{\eps}[X_{\alpha,j},X_{\beta,k}]^{\ell+1})\biggr).\label{eq:Aderivexpansion}
\end{align}
By the Cauchy--Schwarz inequality, we get
\begin{align}
&\mathbb{E}_{\mathbf{X}_{*}^{t,\mathbf{x}_{*}}}(|\mathscr{Q}[\mathbf{X}_{0}]|\mathscr{I}_{t}^{\eps}[X_{\alpha,j},X_{\beta,k}]^{\ell+1}) \notag\\
& \le(\mathbb{E}_{\mathbf{X}^{t,\mathbf{x}}}|\mathscr{Q}[\mathbf{X}]|^{2})^{1/2}(\mathbb{E}_{\mathbf{X}_{*}^{t,\mathbf{x}_{*}}}\mathscr{I}_{t}^{\eps}[X_{\alpha,j},X_{\beta,k}]^{2(\ell+1)})^{1/2}.\label{eq:AderivCauchySchwarz}
\end{align}
Now \cref{lem:logupperbound} tells us that
\begin{align}
(\mathbb{E}_{\mathbf{X}_{*}^{t,\mathbf{x}_{*}}}\mathscr{I}_{t}^{\eps}[X_{\alpha,j},X_{\beta,k}]^{2(\ell+1)})^{1/2}&\le C^{\ell+1}\sqrt{(2(\ell+1))!}|\log\eps|^{\ell+1}\notag\\
&\le(2C)^{\ell+1}(\ell+1)!|\log\eps|^{\ell+1}.\label{eq:Aderivintersections}
\end{align}
 Plugging \cref{eq:AderivCauchySchwarz} and \cref{eq:Aderivintersections}
into \cref{eq:Aderivexpansion} gives us
\[
|Q'(\theta)|\le C_{1}(\mathbb{E}_{\mathbf{X}^{t,\mathbf{x}}}|\mathscr{Q}[\mathbf{X}]|^{2})^{1/2}\sum_{\ell=0}^{\infty}(C_{2}\theta)^{\ell}\le\frac{C_{1}}{1-C_{2}}(\mathbb{E}_{\mathbf{X}^{t,\mathbf{x}}}|\mathscr{Q}[\mathbf{X}]|^{2})^{1/2},
\]
as long as $\theta<C_{2}^{-1}$, for some constants $C_{1}$ and $C_{2}$.
This completes the proof of the lemma.
\end{proof}

\section{Proof of tightness\label{sec:maintheoremproof}}

In this section we prove \cref{thm:maintheorem-quantitative}. The key ingredients
will be the tightness criteria in \cref{thm:tightness-fixedtime,thm:tightness-parabolic} (which we will
apply with $p=2$), the Gaussian Poincar\'e inequality, our Taylor
expansion bound \cref{cor:ginfseriesbound}, and the Brownian motion
intersection estimates in \cref{sec:BMintersectionests}. First we compute
the variance of our KPZ solution integrated against a test function.
\begin{lem}
\label{lem:varcomp}If $\psi$ is a bounded measurable function on $\mathbf{T}^{2}$,
then we have
\begin{align}
&\Var\left(\int h_{\theta}^{\eps}(t,x)\psi(x)\,\dif x\right)\notag\\
&\le\theta\int\int\psi(x_{1})\psi(x_{2})\mathbf{E}\widehat{\mathbb{E}}_{X_{1}^{t,x_{1}},X_{2}^{t,x_{2}}}^{\theta,W,\eps}\mathscr{I}_{t}^{\eps}[X_{1},X_{2}]\,\dif x_{1}\,\dif x_{2}.\label{eq:varbound-1}
\end{align}
\end{lem}

\begin{proof}
By the Gaussian Poincar\'e inequality \cref{eq:GPI}, we have
\[
\Var\left(\int h_{\theta}^{\eps}(t,x)\psi(x)\,\dif x\right)\le\int_{0}^{t}\int\mathbf{E}\left(\int\left(\Dif_{s,y}h_{\theta}^{\eps}(t,x)\right)\psi(x)\,\dif x\right)^{2}\,\dif y\,\dif s.
\]
Now, we have by \cref{eq:malldef} and \cref{eq:detaE}, for $s\in[0,t]$, that
\begin{align}
\Dif_{s,y}h_{\theta}^{\eps}(t,x)&=|\log\eps|^{\frac{1}{2}}\frac{\mathbb{E}_{X^{t,x}}\Dif_{s,y}\mathscr{E}_{\theta,t}^{\eps}[W,X]}{\mathbb{E}_{X^{t,x}}\mathscr{E}_{\theta,t}^{\eps}[W,X]}\notag\\
&=\theta^{\frac{1}{2}}\frac{\mathbb{E}_{X^{t,x}}\rho^{\eps}(X(s)-y)\mathscr{E}_{\theta,t}^{\eps}[W,X]}{\mathbb{E}_{X^{t,x}}\mathscr{E}_{\theta,t}^{\eps}[W,X]}\notag\\
&=\theta^{\frac{1}{2}}\widehat{\mathbb{E}}_{X^{t,x}}^{\theta,W,\eps}\rho^{\eps}(X(s)-y).\label{eq:MalliavinDerivofh}
\end{align}
Therefore, we have
\begin{align*}
&\Var\left(\int h_{\theta}^{\eps}(t,x)\psi(x)\,\dif x\right) \\
  & \le\int_{0}^{t}\int\mathbf{E}\left(\int\theta^{\frac{1}{2}}\widehat{\mathbb{E}}_{X^{t,x}}^{\theta,W,\eps}\rho^{\eps}(X(s)-y)\psi(x)\,\dif x\right)^{2}\,\dif y\,\dif s\\
  & =\theta\int_{0}^{t}\int\int\mathbf{E}\widehat{\mathbb{E}}_{X_{1}^{t,x_{1}},X_{2}^{t,x_{2}}}^{\theta,W,\eps}\biggl(\rho^{\eps}(X_{1}(s)-y)\rho^{\eps}(X_{2}(s)-y)\biggr)\notag\\
 &\qquad \qquad \qquad \qquad \cdot\psi(x_{1})\psi(x_{2})\,\dif x_{1}\,\dif x_{2}\,\dif y\,\dif s\\
  & =\theta\int\int\mathbf{E}\widehat{\mathbb{E}}_{X_{1}^{t,x_{1}},X_{2}^{t,x_{2}}}^{\theta,W,\eps}\mathscr{I}_{t}^{\eps}[X_{1},X_{2}]\psi(x_{1})\psi(x_{2})\,\dif x_{1}\,\dif x_{2},
\end{align*}
which is \cref{eq:varbound-1}.
\end{proof}
Now we derive a bound on the terms of the Taylor-like expansion described in the previous section.
\begin{prop}
\label{prop:fajbkrbound}Fix notation as in \cref{prop:frderivbound-1},
with $J=2$ and
\begin{equation}
\mathscr{Q}[\mathbf{X}_{0}]=\mathscr{I}_{t}^{\eps}[X_{0;1},X_{0;2}].\label{eq:ourQ}
\end{equation}
We have, if $(\alpha,j)\ne(\beta,k)$ and $\eps\le\e^{-t/2}$, that
\[
F_{\alpha,j;\beta,k;r,\mathscr{Q}}(0)\le C^{r+1}(r+1)!|\log\eps|^{r}(t+1+\log|x_{1}-x_{2}|_{\mathbf{T}^{2}}^{-2}).
\]
\end{prop}

\begin{proof}
If $\{(\alpha,j),(\beta,k)\}=\{(0,1),(0,2)\}$, then this is \cref{lem:X0Yest}. So we assume without loss of generality that $(\beta,k)\not\in\{(0,1),(0,2)\}$, since the case $(\alpha,j)\not\in\{(0,1),(0,2)\}$ is similar.
Then note that 
\begin{align*}
&F_{\alpha,j;\beta,k;r,\mathscr{Q}}(0) =\mathbb{E}_{\underline{\mathbf{X}}^{t,\underline{\mathbf{x}}}}(\mathscr{I}_{t}^{\eps}[X_{0;1},X_{0;2}]\mathscr{I}_{t}^{\eps}[X_{\alpha;j},X_{\beta;k}]^{r})\\
 & =\mathbb{E}_{\underline{\mathbf{X}}^{t,\underline{\mathbf{x}}}}\left[\mathscr{I}_{t}^{\eps}[X_{0;1},X_{0;2}]\mathbb{E}_{\underline{\mathbf{X}}^{t,\underline{\mathbf{x}}}}\left(\mathscr{I}_{t}^{\eps}[X_{\alpha;j},X_{\beta;k}]^{r}\,\middle|\,X_{0;1},X_{0;2},X_{\alpha;j}\right)\right]\\
 & \le C^{r}r!|\log\eps|^{r}\mathbb{E}_{\underline{\mathbf{X}}^{t,\underline{\mathbf{x}}}}\mathscr{I}_{t}^{\eps}[X_{0;1},X_{0;2}]\\
 & \le C^{r}r!|\log\eps|^{r}(t+1+\log|x_{1}-x_{2}|_{\mathbf{T}^{2}}^{-2}),
\end{align*}
where the first inequality follows by \cref{lem:logupperbound}  and the second inequality 
by \cref{lem:X0Yest}.
\end{proof}
We are now ready to show our tightness result. 

\begin{proof}[Proof of Theorem \ref{thm:maintheorem-quantitative}]
It follows from \cref{prop:fajbkrbound} 
that, again with the choice of $\mathscr{Q}$ as in \cref{eq:ourQ},
there is a constant $C$ so that, as long as $\eps\le\e^{-t/2}$,
\[
G_{r,\mathscr{Q}}(0)\le C^{r+1}(r+1)!(t+1+\log|x_{1}-x_{2}|_{\mathbf{T}^2}^{-2}).
\]
Combining this with \cref{cor:ginfseriesbound}, we have that
\begin{align}
\mathbf{E}\widehat{\mathbb{E}}_{X_{1}^{t,x_{1}},X_{2}^{t,x_{2}}}^{\theta,W,\eps}\mathscr{I}_{t}^{\eps}[X_{1},X_{2}]&=G_{0, \mathscr{Q}}(\theta)\le\sum_{r=0}^{\infty}\frac{(C\theta)^{r}}{r!}G_{r, \mathscr{Q}}(0) \notag \\
& \le\left(t+1+\log|x_{1}-x_{2}|_{\mathbf{T}^2}^{-2}\right)\sum_{r=0}^{\infty}(C'\theta)^{r}\nonumber \\
 & \le\frac{\theta}{1-C'\theta}(t+1+\log|x_{1}-x_{2}|_{\mathbf{T}^{2}}^{-2}), \label{eq:covbound}
\end{align}
as long as $\theta$ is sufficiently small. Then, using \cref{lem:varcomp}
and \cref{eq:covbound}, and identifying $\mathbf{T}^2$ with $(-1/2,1/2]^2$ in the third and fourth lines below, we have for any $\psi\in\mathcal{C}_c((-1/2, 1/2)^{2})$,
\begin{align}
&\mathbf{E}\left|\int h_{\theta}^{\eps}(t,x)\mathcal{S}^{2^{-n}}\psi(x)\,\dif x\right|^{2} \notag\\
& \le\theta\int\int\mathcal{S}^{2^{-n}}\psi(x_{1})\mathcal{S}^{2^{-n}}\psi(x_{2})\mathbf{E}\widehat{\mathbb{E}}_{X_{1}^{t,x_{1}},X_{2}^{t,x_{2}}}^{\theta,W,\eps}\mathscr{I}_{t}^{\eps}[X_{1},X_{2}]\,\dif x_{1}\,\dif x_{2}\nonumber \\
&= 2^{4n}\theta\int_{\mathbf{R}^2}\int_{\mathbf{R}^2}\psi(2^{n}x_{1})\psi(2^{n}x_{2})\mathbf{E}\widehat{\mathbb{E}}_{X_{1}^{t,x_{1}},X_{2}^{t,x_{2}}}^{\theta,W,\eps}\mathscr{I}_{t}^{\eps}[X_{1},X_{2}]\,\dif x_{1}\,\dif x_{2}\nonumber \\
 & \le\frac{\theta^{2}}{1-C'\theta}\int_{\mathbf{R}^2}\int_{\mathbf{R}^2}\psi(x_{1})\psi(x_{2})(t+1+2n\log2+\log|x_{1}-x_{2}|_{\mathbf{T}^{2}}^{-2})\,\dif x_{1}\,\dif x_{2}\nonumber \\
 & \le\frac{\theta^{2}}{1-C'\theta}\|\psi\|_{L^\infty}^{2}\left(t+1+2n\log2+\int\int\log|x_{1}-x_{2}|_{\mathbf{T}^{2}}^{-2}\,\dif x_{1}\,\dif x_{2}\right),\label{eq:final-varbound-1}
\end{align}
which proves that $\{h_{\theta}^{\eps}(t,\cdot)\}_{\eps>0}$ is tight
in $\mathcal{C}^{-1-\delta}(\mathbf{T}^{2})$ for any $\delta>0$, by \cref{thm:tightness-fixedtime} (with $p=2$, $\beta < 0$ arbitrary, and $\alpha < \beta-1$ arbitrary). Next,
for any $\psi\in\mathcal{C}_c((-1/2, 1/2)^{2})$, $\phi\in\mathcal{C}_c(\mathbf{R}_{>0})$, $n\ge k\ge 1$, $t>2^{-2k}$, and $\xi(t,x) = \phi(t)\psi(x)$, the Cauchy--Schwarz inequality gives
\begin{align*}
&\mathbf{E} \left|\int_0^\infty\int h_{\theta}^{\eps}(s,x)\mathcal{S}^{2^{-n}}_{\mathfrak{s}}\xi(t-s, x)\,\dif x\,\dif s\right|^{2}\\
 & \le2^{8n}\left(\int_0^\infty|\phi(2^{2n}(t-s))|\,\dif s\right)\\
 &\qquad \qquad \cdot \int_0^\infty|\phi(2^{2n}(t-s))|\mathbf{E}\left|\int_{\mathbf{R}^2} h_{\theta}^{\eps}(s,x)\psi(2^{n}x)\,\dif x\right|^{2}\,\dif s,
 \end{align*}
 where, as before, we identified $\mathbf{T}^2$ with $(-1/2,1/2]^2$ in the last line. 
Now,
\begin{align*}
\int_0^\infty|\phi(2^{2n}(t-s))|\,\dif s &\le 2^{-2n} \|\phi\|_{L^1}. 
\end{align*}
On the other hand, by \cref{eq:final-varbound-1}, 
\begin{align*}
&\mathbf{E}\left|\int_{\mathbf{R}^2} h_{\theta}^{\eps}(s,x)\psi(2^{n}x)\,\dif x\right|^2 \\
&\le \frac{2^{-4n}\theta^{2}}{1-C'\theta}\|\psi\|_{L^\infty}^{2}\left(s+1+n\log2+\int\int\log|x_{1}-x_{2}|_{\mathbf{T}^{2}}^{-2}\,\dif x_{1}\,\dif x_{2}\right).
\end{align*}
Thus, 
\begin{align*}
&\int_0^\infty|\phi(2^{2n}(t-s))|\mathbf{E}\left|\int_{\mathbf{R}^2} h_{\theta}^{\eps}(s,x)\psi(2^{n}x)\,\dif x\right|^{2}\,\dif s\\
&\le \frac{2^{-4n}\theta^{2}}{1-C'\theta}\|\psi\|_{L^\infty}^{2} \int_0^\infty |\phi(2^{2n}(t-s))|\biggl(s+1+2n\log2\\
&\qquad \qquad +\int\int\log|x_{1}-x_{2}|_{\mathbf{T}^{2}}^{-2}\,\dif x_{1}\,\dif x_{2}\biggr)\, \dif s.
\end{align*}
But since $\phi\in \mathcal{C}_c(\mathbf{R}_{>0})$, 
\begin{align*}
\int_0^\infty |\phi(2^{2n}(t-s))|s\, \dif s &= 2^{-2n} \int_0^{2^{2n}t}|\phi(u)| (t-2^{-2n}u)\, \dif u\\
&\le 2^{-2n} C(\phi)t,
\end{align*}
where $C(\phi)$ depends only on $\phi$. Similarly,
\[
\int_0^\infty |\phi(2^{2n}(t-s))|\, \dif s \le 2^{-2n} C(\phi).
\]
Combining these observations, we see that
\begin{align*}
&\mathbf{E} \left|\int_0^\infty\int h_{\theta}^{\eps}(s,x)\mathcal{S}^{2^{-n}}_{\mathfrak{s}}\xi(t-s,x)\,\dif x\,\dif s\right|^{2}\notag\\
 & \le\frac{\theta^{2}}{1-C'\theta}C(\phi, \psi)\biggl(t+1\notag \\
 &\qquad \qquad +2n\log2+\int\int\log|x_{1}-x_{2}|_{\mathbf{T}^{2}}^{-2}\,\dif x_{1}\,\dif x_{2}\biggr),
\end{align*}
where $C(\phi, \psi)$ depends only on $\phi$ and $\psi$. 
In light of \cref{thm:tightness-parabolic} (with $p=2$, $\beta<0$ arbitrary and $\alpha < \beta-2$ arbitrary) and \cref{rem:prodstruct}, this shows
that $\{h_{\theta}^{\eps}\}_{\eps>0}$ is tight in $\mathcal{C}_{\mathfrak{s};\mathrm{loc}}^{-2-\delta}(\mathbf{R}_{>0}\times\mathbf{T}^{2})$ for any $\delta>0$.
\end{proof}

\section{Non-vanishing effect of the nonlinearity\label{sec:nontriviality}}
In this section we prove \cref{thm:limitsnontrivial}: that the zeroeth Fourier mode of our limiting KPZ solution has a different law than the zeroeth Fourier mode of the additive stochastic heat equation with the same noise strength. To do this, we notice that the zeroeth Fourier mode of the solution to the ASHE only sees contributions from the zeroeth Fourier mode of the white noise, since the Fourier transform diagonalizes the Laplacian. Because the KPZ nonlinearity depends only on the derivative of the solution, and thus does not see the zeroeth Fourier mode, so the contributions from the zeroeth Fourier mode of the noise to the zeroeth Fourier mode of the solution are the same for the KPZ solution as they are for the ASHE solution. On the other hand, as we will show in this section, the KPZ nonlinearity \emph{does} make contributions from \emph{higher} Fourier modes of the white noise to the zeroeth Fourier mode of the solution. These extra contributions are what distinguishes the KPZ solution from the ASHE solution.

In \cref{lem:reduction} below, we will formalize the idea that the added nonlinear contributions to the zeroeth Fourier mode from higher Fourier modes of the noise will distinguish the KPZ solution from the ASHE solution. In the remainder of this section, we will show that these contributions exist for positive $\eps$ and do not vanish as $\eps\to 0$. In passing to the limit, the elementary \cref{lem:nontrivialitycondition} below will play an important role.

For $\xi\in L^{2}([0,t]\times\mathbf{T}^{2};\mathbf{C})$,
define
\begin{equation}
W_{t}[\xi]=\int_{0}^{t}\int\xi(s,x)W(\dif x\,\dif s).\label{eq:Zdef}
\end{equation}
We note that $W_{t}[\xi]$ is a Gaussian random variable, that $W_{t}[\overline{\xi}]=\overline{W_{t}[\xi]}$,
and that
\begin{equation}
\mathbf{E}W_{t}[\xi]W_{t}[\eta]=\int_{0}^{t}\int\xi(s,x)\eta(s,x)\,\dif x\,\dif s.\label{eq:covariance}
\end{equation}
Therefore, if $\{\xi_{\alpha}\}_{\alpha}$ is an orthonormal set in
$L^{2}([0,t]\times\mathbf{T}^{2};\mathbf{C})$, then $\{W_{t}[\xi_{\alpha}]\}_{\alpha}$
is an orthonormal set in $L^{2}(\Omega,\mathcal{F},\mathbf{P};\mathbf{C})$.
Moreover, if $\{\xi_{\alpha}\}_{\alpha}$ is an orthonormal set in
$L^{2}([0,t]\times\mathbf{T}^{2};\mathbf{C})$ which also satisfies\begin{equation}
\int_{0}^{t}\int\xi_{\alpha}(s,x)\xi_{\beta}(s,x)\,\dif x\,\dif s=\delta_{\alpha,\beta},\label{eq:pseudocovariance}
\end{equation} 
then $\{W_{t}[\xi_{\alpha}]\}_{\alpha}$ is a collection of independent
complex Gaussian random variables. (Note that the difference between \cref{eq:pseudocovariance} and orthogonality in $L^{2}([0,t]\times\mathbf{T}^{2};\mathbf{C})$ is that no complex conjugate is taken in \cref{eq:pseudocovariance}.)

Recall the function $v_{\theta}$ solving the additive stochastic
heat equation \cref{eq:htildePDE-1}. We consider the spatial Fourier
transform at $0$, 
\[
\widehat{v_{\theta}}(t,0)=\int v_{\theta}(t,x)\,\dif x.
\]
Taking the Fourier transform of  \cref{eq:htildePDE-1} and evaluating at $0$, we see that $\widehat{v_{\theta}}(t,0)$
solves the stochastic differential equation
\begin{equation}
\dif\widehat{v_{\theta}}(t,0)=\sqrt{\theta}\dif W_{t}[1],\label{eq:vbarprob}
\end{equation}
and so
\[
\widehat{v_{\theta}}(t,0)=\sqrt{\theta}W_{t}[1].
\]
Now define
\[
\widetilde{W}_t=W_t-W_{t}[1].
\]
Here, of course, $W_t$ denotes the cylindrical Wiener process at time $t$. Analogously to \cref{eq:Zdef}, define
\[
 \widetilde{W}_t[\xi] = W_t[\xi]-\int_0^t\int \xi(s,y)\,\dif y\,\dif W_s[1].
\]
Now we have, using \cref{eq:covariance}, that
\[
\mathbf{E} \widetilde{W}_t[\xi]W_s[1] = \mathbf{E} \left(W_t[\xi]-\int_0^t\int \xi(s,y)\,\dif y\,\dif W_t[1]\right)W_s[1]=0,
\]
for all $\xi$, so $\widetilde{W}$ and $\{W_{t}[1]\}_{t\ge0}$ are
independent. Also, note that $W_t[1]$ is constant in space, so mollifying it has no effect. Thus,
\begin{equation}
\begin{cases}
\partial_{t}\widetilde{h}_{\theta}^{\eps}(t,x)=\frac{1}{2}\Delta\widetilde{h}_{\theta}^{\eps}(t,x)+\frac{1}{2}|\log\eps|^{-\frac{1}{2}}|\nabla\widetilde{h}_{\theta}^{\eps}(t,x)|^{2}&\\
\qquad \qquad \qquad +\sqrt{\theta}\dot{W}_{t}[1]+\sqrt{\theta}\dot{\widetilde{W}}^{\eps}(t,x), & t>0,x\in\mathbf{T}^{2},\\
\widetilde{h}_{\theta}^{\eps}(0,x)=0, & x\in\mathbf{T}^{2}.
\end{cases}\label{eq:htildesplitup}
\end{equation}
Now we define a function
\begin{equation}\label{eq:fdef}\widetilde{f}_\theta^\eps(t,x)=\widetilde{h}_\theta^\eps(t,x)-\widehat{v_\theta}(t,0),\end{equation}
so by \cref{eq:vbarprob} and \cref{eq:htildesplitup},
$\widetilde{f}_\theta^\eps(t,x)$ solves the SPDE
\[
\begin{cases}
\partial_{t}\widetilde{f}_{\theta}^{\eps}(t,x)=\frac{1}{2}\Delta\widetilde{f}_{\theta}^{\eps}(t,x)+\frac{1}{2}|\log\eps|^{-\frac{1}{2}}|\nabla\widetilde{f}_{\theta}^{\eps}(t,x)|^{2}&\\
\qquad \qquad \qquad +\sqrt{\theta}\dot{\widetilde{W}}^{\eps}(t,x), & t>0,x\in\mathbf{T}^{2},\\
\widetilde{f}_{\theta}^{\eps}(0,x)=0, & x\in\mathbf{T}^{2}.
\end{cases}
\]
Combining \cref{eq:fdef} with \cref{eq:fkform}, we can derive the expression
\begin{align}
 \widetilde{f}_\theta^\eps(t,x)&=|\log\eps|^{\frac{1}{2}}\log\mathbb{E}_{X^{t,x}}\exp\biggl\{ \theta^{\frac{1}{2}}|\log\eps|^{-\frac{1}{2}}\notag\\
&\qquad \qquad \qquad \qquad\qquad\qquad \cdot\int_{0}^{t}\int\rho^{\eps}(X(s)-y)\widetilde{W}(\dif y\,\dif s)\biggr\},\label{eq:ffk}
\end{align}
Since $\widetilde{W}$ and $\{W_t[1]\}_t$ are independent processes,
we can conclude from \cref{eq:vbarprob} and \cref{eq:ffk} that
$\widetilde{f}_{\theta}^{\eps}$ and $\widehat{v_{\theta}}(\cdot,0)$
are independent. Now define
\[
f_{\theta}^{\eps}(t,x)=\widetilde{f}_{\theta}^{\eps}(t,x)-\kappa_{\theta}^{\eps}(t),
\]
so that
\begin{equation}
h_{\theta}^{\eps}(t,x)=f_{\theta}^{\eps}(t,x)+\widehat{v_{\theta}}(t,0)\label{eq:hfv}
\end{equation}
and $f_{\theta}^{\eps}$ and $\widehat{v_{\theta}}(\cdot,0)$ are
independent. Moreover, since 
\begin{equation}\label{eq:hvexp0}
\mathbf{E}h_{\theta}^{\eps}(t,x)=\mathbf{E}\widehat{v_{\theta}}(t,0)=0,
\end{equation}
we have $\mathbf{E}f_{\theta}^{\eps}(t,x)=0$ as well.

Our primary goal in this section is to show that there is no sequence $\eps_{n}\downarrow0$
so that $\int h_{\theta}^{\eps_{n}}(t,x)\,\dif x$ converges in distribution
to $\widehat{v_{\theta}}(t,0)$. Without loss of generality, we will work with $t=1$ throughout. The proof for general $t$ is similar. We begin with the following reduction.
\begin{lem}\label{lem:reduction}
If
\[
\int h_{\theta}^{\eps_{n}}(1,x)\,\dif x\to\widehat{v_{\theta}}(1,0)
\]
in law, then
\[
\int f_{\theta}^{\eps_n}(1,x)\,\dif x\to0
\]
in probability.
\end{lem}

\begin{proof}
We first note that, by \cref{eq:final-varbound-1}, we have that $\left\{ \int h_{\theta}^{\eps}(1,x)\,\dif x\right\}_{\eps>0}$
is uniformly bounded in $L^{2}$, so by \cref{eq:hfv} and the fact
that $\widehat{v_{\theta}}(1,0)$ is Gaussian, $\left\{ \int f_{\theta}^{\eps}(1,x)\,\dif x\right\}_{\eps>0}$
is uniformly bounded in $L^{2}$ as well. Therefore, 
\[
\left\{ \left(\int h_{\theta}^{\eps}(1,x)\,\dif x,\int f_{\theta}^{\eps}(1,x)\,\dif x,\, \widehat{v_{\theta}}(1,0)\right)\right\}_{\eps>0}
\]
converges in law along subsequences. Now suppose that there is
a sequence $\eps_{n}\downarrow0$ so that 
\begin{equation}
\int h_{\theta}^{\eps_{n}}(1,x)\,\dif x\to\widehat{v_{\theta}}(1,0)\label{eq:htov}
\end{equation}
in law. Possibly replacing $(\eps_{n})$ by a subsequence,
we can assume that
\[
\left(\int h_{\theta}^{\eps_{n}}(1,x)\,\dif x,\int f_{\theta}^{\eps_{n}}(1,x)\,\dif x,\, \widehat{v_{\theta}}(1,0)\right)\to(H,F,\widehat{v_{\theta}}(1,0))
\]
in law for some random variables $H$ and $F$. Since $\int f_{\theta}^{\eps_{n}}(1,x)\,\dif x$
and $\widehat{v_{\theta}}(1,0)$ are independent for each $\eps_{n}$,
we must also have that $F$ and $\widehat{v_{\theta}}(1,0)$ are independent.
But by \cref{eq:hfv}, we must have $H=F+\widehat{v_{\theta}}(1,0)$,
so by \cref{eq:htov} we must have that $\widehat{v_{\theta}}(1,0)$
is equal in distribution to $F+\widehat{v_{\theta}}(1,0)$. This means
that the characteristic function of $F$ must be equal to $1$ on
the support of the characteristic function of $\widehat{v_{\theta}}(1,0)$,
which is all of $\mathbf{R}$ since $\widehat{v_{\theta}}(1,0)$ is
Gaussian. Thus, the characteristic function of $F$ must be identically
$1$, and so $F$ must be deterministically equal to $0$. Therefore 
$\int f_{\theta}^{\eps_{n}}(1,x)\,\dif x$ converges in law
to the point mass at $0$, and hence converges in probability to $0$.
\end{proof}
Therefore, to prove \cref{thm:limitsnontrivial} it is sufficient to
prove the following theorem.
\begin{thm}
\label{thm:fdoesntgoto0}There is no sequence $\eps_{n}\downarrow0$
so that $\int f_{\theta}^{\eps_{n}}(1,x)\,\dif x\to0$ in probability. Moreover, there is no sequence 
$\eps_{n}\downarrow0$
so that $\int h_{\theta}^{\eps_{n}}(1,x)\,\dif x\to0$ in probability.
\end{thm}

(We also have to show that subsequential limits of $\int f^\eps_\theta(1,x)\, \dif x$ and $\int h^\eps_\theta(1,x)\, \dif x$ have mean zero. But this is easy because these random variables have mean zero and are uniformly bounded in $L^2$ by \cref{eq:final-varbound-1} and \cref{eq:hfv}.)

We will prove \cref{thm:fdoesntgoto0} at the end of this section. Our
strategy will be to show that the projection of $f_{\theta}^{\eps}$
onto the second Wiener chaos has $L^{2}$ norm which is
not going to $0$ with $\eps$. (See, for example, \cite{Jan97} for background
on the Wiener chaos decomposition.) To show that this is sufficient,
we will need the following lemma.
\begin{lem}
\label{lem:nontrivialitycondition}Suppose that $\{A_{n}\},\{B_{n}\}$
are sequences of random variables defined on the same probability space 
and assume that the following conditions hold:
\begin{enumerate}
\item[$(1)$] $\mathbf{E}A_{n}=\mathbf{E}A_{n}B_{n}=0$ for each $n$.
\item[$(2)$] There is a constant $c>0$ so that $\mathbf{E}A_{n}^{2}\ge c$ for
each $n$.
\item[$(3)$] There is a constant $C>0$ and a constant $p>2$ so that $\mathbf{E}|A_{n}|^{p}$ and $\mathbf{E}B_{n}^{2}$ are bounded by $C$
for each $n$.
\end{enumerate}
Then $A_{n}+B_{n}$ cannot converge in probability to $0$.
\end{lem}

This lemma is an exercise in elementary probability theory. For completeness,
we include its proof in \cref{subsec:prob-proofs}. Now we begin the
proof of \cref{thm:fdoesntgoto0} in earnest. We start by writing an
expression for the coefficients of the relevant elements of the second
Wiener chaos in the decomposition of $\int f_{\theta}^{\eps}(1,x)\,\dif x$.
Define, for $k\in\mathbf{Z}^{2}$, $\ell\in\{0,\ldots,|k|^{2}-1\}$, $s\in [0,1]$, $y\in\mathbf{T}^2$,
\[
\e_{k,\ell}(s,y)=\exp\{2\pi\ii(\ell s+k\cdot y)\}.
\]
In the following, let $\mathbf{X}=(X_{1},X_{2})$ and $\mathbf{0}=(0,0)$ where 
$0=(0,0)\in\mathbf{T}^{2}$. Let $\widehat{\rho}$ be the Fourier transform
of $\rho$ (considered as a function from $\mathbf{R}^{2}$ into $\mathbf{R}$). For a path $X$, let 
\[
\mathscr{S}_{k,\ell}[X]=\int_{0}^{1}\e_{k,\ell}(s,X(s))\,\dif s.
\]
Define
\begin{equation}
\mathscr{A}_{k,\ell}[\mathbf{X}]=\int\left(|\mathscr{S}_{k,\ell}[X_{1}+x]|^{2}-\mathscr{S}_{k, \ell}[X_{1}+x]\overline{\mathscr{S}_{k,\ell}[X_{2}+x]}\right)\,\dif x.\label{eq:Axiexpr}
\end{equation}
Lastly, let
\begin{equation}\label{eq:adef}
a_{\theta;k,\ell}^{\eps}=\mathbf{E}\biggl[(W_{1}[\e_{k,\ell}]\overline{W_{1}[\e_{k,\ell}]}-1)\int f_{\theta}^{\eps}(1,x)\,\dif x\biggr]
\end{equation}
\begin{lem}
If $k\ne0$, then 
\begin{equation}
a_{\theta;k,\ell}^{\eps}=\theta|\widehat{\rho}(\eps k)|^{2}|\log\eps|^{-\frac{1}{2}}\mathbf{E}\widehat{\mathbb{E}}_{\mathbf{X}^{1,\mathbf{0}}}^{\theta,W,\eps}\mathscr{A}_{k,\ell}[\mathbf{X}].\label{eq:Aexpr}
\end{equation}
\end{lem}

\begin{proof}
We first note that since $k\ne0$, $\int\overline{\e_{k,\ell}(t,x)}\,\dif x=\int\e_{k,\ell}(t,x)\,\dif x=0$
for all $t$. This implies that the random variables $W_{1}[\e_{k,\ell}]$ and $\widehat{v_{\theta}}(1,0)$
are independent (recall the discussion surrounding \cref{eq:pseudocovariance}), and therefore $W_{1}[\e_{k,\ell}]\overline{W_{1}[\e_{k,\ell}]}-1$
and $\widehat{v_{\theta}}(1,0)$ are independent. By \cref{eq:hfv} and \cref{eq:hvexp0}
this means that 
\begin{align*}
a_{\theta;k,\ell}^{\eps}&=\mathbf{E}\biggl[(W_{1}[\e_{k,\ell}]\overline{W_{1}[\e_{k,\ell}]}-1)\int h_{\theta}^{\eps}(1,x)\,\dif x\biggr]\\
&= \mathbf{E}\biggl[W_{1}[\e_{k,\ell}]\overline{W_{1}[\e_{k,\ell}]}\int h_{\theta}^{\eps}(1,x)\,\dif x\biggr]\\
&= \mathbf{E}\biggl[\biggl(\int_0^1 \int \e_{k,\ell}(s,y) W(\dif y \,\dif s) \biggr)\biggl(\overline{W_{1}[\e_{k,\ell}]}\int h_{\theta}^{\eps}(1,x)\,\dif x\biggr)\biggr].
\end{align*}
By the Gaussian integration by parts formula \cref{eq:GIP}, this gives
\begin{align*}
a_{\theta;k,\ell}^{\eps}&=\mathbf{E}\biggl[\int_0^1\int \e_{k,\ell}(s,y) \Dif_{s,y}\biggl(\overline{W_{1}[\e_{k,\ell}]}\int h_{\theta}^{\eps}(1,x)\,\dif x\biggr)\,\dif y\,\dif s\biggr].
\end{align*}
By the product rule for the Malliavin derivative and \cref{eq:malldef},
\begin{align*}
&\Dif_{s,y}\biggl(\overline{W_{1}[\e_{k,\ell}]}\int h_{\theta}^{\eps}(1,x)\,\dif x\biggr)\\
&= (\Dif_{s,y}\overline{W_{1}[\e_{k,\ell}]}) \int h_{\theta}^{\eps}(1,x)\,\dif x + \overline{W_{1}[\e_{k,\ell}]} \int \Dif_{s,y} h_{\theta}^{\eps}(1,x)\,\dif x\\
&= \overline{e_{k,\ell}(s,y)}\int h_{\theta}^{\eps}(1,x)\,\dif x + \overline{W_{1}[\e_{k,\ell}]} \int \Dif_{s,y} h_{\theta}^{\eps}(1,x)\,\dif x.
\end{align*}
Therefore, again applying \cref{eq:hvexp0}, we get
\begin{align*}
a_{\theta;k,\ell}^{\eps}&=\mathbf{E} \biggl(\overline{W_{1}[\e_{k,\ell}]}\int_0^1\int\int \e_{k,\ell}(s,y) \Dif_{s,y}h_{\theta}^{\eps}(1,x)\,\dif x\,\dif y\,\dif s\biggr). 
\end{align*}
By the formula \cref{eq:MalliavinDerivofh} for $\Dif_{s,y}h_{\theta}^{\eps}(1,x)$, this shows that
\begin{align*}
a_{\theta;k,\ell}^{\eps} &= \theta^{\frac{1}{2}}\mathbf{E}\biggl[ \overline{W_{1}[\e_{k,\ell}]}\int\widehat{\mathbb{E}}_{X^{1,x}}^{\theta,W,\eps}\biggl(\int_0^1\int\rho^{\eps}(X(s)-y)\e_{k,\ell}(s,y)\,\dif y\,\dif s\biggr)\,\dif x\biggr].
\end{align*}
Since $\rho^\eps$ is an even function,
\begin{align}
&\int_0^1\int\rho^{\eps}(X(s)-y)\e_{k,\ell}(s,y)\,\dif y\,\dif s\notag \\
&= \int_0^1\int\rho^{\eps}(X(s)-y)e^{2\pi\ii(\ell s+k\cdot (y-X(s)))}e^{2\pi \ii k\cdot X(s)}\,\dif y\,\dif s\notag\\
&= \int_0^1 e^{2\pi\ii(\ell s+k\cdot X(s))}\biggl(\int\rho^{\eps}(X(s)-y)e^{2\pi\ii k\cdot (y-X(s))}\,\dif y\biggr)\,\dif s\notag\\
&= \int_0^1 e^{2\pi\ii(\ell s+k\cdot X(s))}\biggl(\int\rho^{\eps}(z)e^{2\pi\ii k\cdot z}\,\dif z\biggr)\,\dif s \notag\\
&= \widehat{\rho^\eps}(k) \mathscr{S}_{k,\ell}[X] = \widehat{\rho}(\eps k) \mathscr{S}_{k,\ell}[X].\label{eq:rhoform}
\end{align}
Combining all of the above, we get
\[
a_{\theta;k,\ell}^{\eps}
= \theta^{\frac{1}{2}}\widehat{\rho}(\eps k)\mathbf{E} \biggl(\overline{W_{1}[\e_{k,\ell}]}\int\widehat{\mathbb{E}}_{X^{1,x}}^{\theta,W,\eps}\mathscr{S}_{k,\ell}[X]\,\dif x\biggr).
\]
Now we can integrate by parts again, and use \cref{eq:malliavinderiv} and \cref{eq:rhoform},  to obtain
\begin{align*}
 a_{\theta;k,\ell}^{\eps} &= \theta^{\frac{1}{2}}\widehat{\rho}(\eps k)\mathbf{E}\biggl[ \int_0^1\int \overline{\e_{k,\ell}(s,y)}\biggl(\int\Dif_{s,y}\widehat{\mathbb{E}}_{X^{1,x}}^{\theta,W,\eps}\mathscr{S}_{k,\ell}[X]\,\dif x\biggr)\,\dif y\,\dif s\biggr]\notag\\
&= \theta\widehat{\rho}(\eps k)|\log\eps|^{-\frac{1}{2}}\mathbf{E}\biggl[ \int_0^1\int \overline{\e_{k,\ell}(s,y)}\biggl(\int\widehat{\mathbb{E}}_{X^{1,x},\widetilde{X}^{1,x}}^{\theta,W,\eps}\mathscr{S}_{k,\ell}[X]\notag\\&\qquad\qquad\qquad \qquad \qquad\cdot (\rho^\eps(X(s)-y)-\rho^\eps(\widetilde{X}(s)-y))\,\dif x\biggr)\,\dif y\,\dif s\biggr]\notag\\
&= \theta\widehat{\rho}(\eps k)|\log\eps|^{-\frac{1}{2}}\mathbf{E}\biggl[ \int\widehat{\mathbb{E}}_{X^{1,x},\widetilde{X}^{1,x}}^{\theta,W,\eps}\mathscr{S}_{k,\ell}[X](\overline{\mathscr{S}_{k,\ell}[X]}-\overline{\mathscr{S}_{k,\ell}[\widetilde{X}]})\,\dif x\biggr],
\end{align*}
which is \cref{eq:Aexpr}.
\end{proof}
Now define
\begin{equation}
\widetilde{a}_{\theta;k,\ell}^{\eps}=|\log\eps|^{-\frac{1}{2}}|\widehat{\rho}(\eps k)|^{2}\mathbf{E}\widehat{\mathbb{E}}_{\mathbf{X}^{1,\mathbf{x}}}^{\theta,W,\eps}\mathscr{A}_{k,\ell}[\mathbf{X}],\label{eq:Qtildedef}
\end{equation}
so that
\begin{equation}
a_{\theta;k,\ell}^{\eps}=\theta\widetilde{a}_{\theta;k,\ell}^{\eps}.\label{eq:QQtilde}
\end{equation}
We want to lower-bound $a_{\theta;k,\ell}^{\eps}$, which we will
achieve by lower bounding $\widetilde{a}_{0;k,\ell}^{\eps}$ and upper bounding
the derivative of $\widetilde{a}_{\theta;k,\ell}^{\eps}$ with respect to $\theta$.
Our tool for the latter purpose will be \cref{prop:simpleUB}. Thus
we first need to prove some estimates on the quantities involved in
\cref{eq:Qtildedef} with $\theta=0$, and on the terms involved in
\cref{eq:dAtildebound} with the choice $\mathscr{Q}=\mathscr{A}_{k,\ell}^{\eps}$.
\begin{lem}
\label{lem:Mp}Define 
\begin{equation}
M_{k,\ell;2p}\coloneqq\mathbb{E}_{X^{1,0}}\int|\mathscr{S}_{k,\ell}[X+x]|^{2p}\,\dif x.\label{eq:Mpdef}
\end{equation}
Then for any integer $p\ge1$ there is a constant $C_p$ such that for any $\ell$ and any $k\ne 0$, 
\begin{equation}
M_{k,\ell;2p}\le C_p|k|^{-2p}.\label{eq:upperbound}
\end{equation}
Moreover, there is an absolute constant $c>0$ such that if $k\ne 0$ and $|\ell|\le|k|^{2}$, then
\begin{equation}
M_{k,\ell;2}\ge c|k|^{-2}.\label{eq:secondmoment}
\end{equation}
\end{lem}

We will only use \cref{eq:upperbound} in the case $p=2$. Since the
right-hand side of \cref{eq:Mpdef} can be evaluated explicitly (although
perhaps only a computer algebra system would have the patience), we
present a slightly long but ultimately straightforward computational
proof of \cref{lem:Mp} in \cref{subsec:BMcomps}. On the other hand,
it is easy to interpret the order of magnitude of the fluctuations
of $\mathscr{S}_{k,\ell}[X]$ probabilistically. The integral in the
definition of $\mathscr{S}_{k,\ell}[X]$ sums the values of a sinusoid
with frequency $|k|$ at the position of a Brownian motion, and it
takes the Brownian motion time $|k|^{2}$ to move a distance $|k|$,
so the integral is effectively averaging $|k|^{-2}$ i.i.d. random
variables. Hence the fluctuations of $\mathscr{S}_{k,\ell}[X]$ are
on the order $|k|^{-1}$.
\begin{lem}
For any $x\in \mathbf{T}^2$, any $k\ne 0$ and any $\ell$, 
\begin{equation}
|\mathbb{E}_{X^{1,0}}\mathscr{S}_{k,\ell}[X]|\le2|k|^{-2}.\label{eq:explowerbound}
\end{equation}
\end{lem}

\begin{proof}
Suppose that $B$ is a Brownian motion on $\mathbf{R}^2$ started from the origin at time $t$ and flowing backwards in time. Let $X$ be the projection of  $B$ on to the torus $\mathbf{T}^2$, so that $X$ is a Brownian motion on the torus. Then for any $s\le t$, $X(s) - B(s) \in \mathbf{Z}^2$, and therefore for any $s\le t$ and any $k\in \mathbf{Z}^2$,  $e^{2\pi \ii k\cdot X(s)} = e^{2\pi \ii k \cdot B(s)}$. We will use this fact in this proof and also later. One immediate consequence is that 
\begin{equation}\label{eq:brownian}
\mathbf{E}(e^{2\pi \ii k\cdot X(s)}) = \mathbf{E}(e^{2\pi \ii k \cdot B(s)}) = e^{-2\pi^2 |k|^2s}. 
\end{equation}
Using this, we compute
\begin{align*}
&\mathbb{E}_{X^{1,0}}\int_{0}^{1}\e_{k,\ell}(s,X(s)+x)\,\dif s\notag\\
&=\int_{0}^{1}\mathbb{E}_{X^{1,0}}\exp\{2\pi\ii(\ell s+k\cdot X(s)+k\cdot x)\}\,\dif s\notag\\
&=\int_{0}^{1}\exp\{2\pi\ii(\ell s+k\cdot x)-2\pi^{2}|k|^{2}s\}\,\dif s\\
&=\e^{2\pi\ii k\cdot x}\frac{1-\e^{2\pi\ii\ell-2\pi|k|^{2}s}}{-2\pi\ii\ell+2\pi^{2}|k|^{2}}. 
\end{align*}
Since the absolute value of the numerator is clearly bounded by $2$, and $|a+\ii b|\ge |a|$ for any $a,b\in \mathbf{R}$, this proves \cref{eq:explowerbound}.
\end{proof}
\begin{lem}\label{lem:abound}
There is a $\theta_{2}>0$ and constants $c,k_{0}>0$ so that, if
$0\le\theta\le\theta_{2}$, $|k|\ge k_{0}$, and $\ell\le|k|^{2}$,
then
\begin{equation}
a_{\theta;k,\ell}^{\eps}\ge c\theta|\log\eps|^{-\frac{1}{2}}|\widehat{\rho}(\eps k)|^{2}|k|^{-2}.\label{eq:Alowerbound}
\end{equation}
\end{lem}

\begin{proof}
By the mean value theorem, we have
\[
\widetilde{a}_{\theta;k,\ell}^{\eps}\ge\widetilde{a}_{0;k,\ell}^{\eps}-\theta\max_{\theta'\in[0,\theta]}\left|\frac{\partial}{\partial\theta}\widetilde{a}_{\theta;k,\ell}^{\eps}\right|.
\]
Now,
\begin{align*}
&\mathbb{E}_{\mathbf{X}^{1,\mathbf{0}}}\mathscr{A}_{k,\ell}[\mathbf{X}] \\
&= \int\mathbb{E}_{\mathbf{X}^{1,\mathbf{0}}}\left(|\mathscr{S}_{k,\ell}[X_{1}+x]|^{2}-\mathscr{S}_{k, \ell}[X_{1}+x]\overline{\mathscr{S}_{k,\ell}[X_{2}+x]}\right)\,\dif x\\
&= M_{k,\ell; 2} - \int (\mathbb{E}_{X^{1,0}}\mathscr{S}_{k,\ell}[X])^2 \, \dif x.
\end{align*}
Therefore it follows from \cref{eq:secondmoment} and \cref{eq:explowerbound} that
there is some $k_{0}>0$ and some $c>0$ so that, as long as $|k|\ge k_{0}$ and $\ell\le|k|^{2}$,
we have
\[
\mathbb{E}_{\mathbf{X}^{1,\mathbf{0}}}\mathscr{A}_{k,\ell}[\mathbf{X}]\ge c|k|^{-2}.
\]
Thus (recalling \cref{eq:Qtildedef}), we have that
\[
\widetilde{a}_{0;k,\ell}^{\eps}\ge c|\log\eps|^{-\frac{1}{2}}|\widehat{\rho}(\eps k)|^{2}|k|^{-2}.
\]
Moreover, we can use the Cauchy--Schwarz inequality on \cref{eq:Axiexpr}
and then apply \cref{eq:upperbound} to write
\[
\mathbb{E}_{\mathbf{X}^{1,\mathbf{0}}}\mathscr{A}_{k,\ell}[\mathbf{X}]^{2}\le4\mathbb{E}_{X^{1,0}}\int|\mathscr{S}_{k,\ell}[X_{1}+x]|^{4}\,\dif x\le C|k|^{-4}.
\]
This means that, by \cref{eq:dAtildebound}, we have, as long as $\theta<\theta_{0}$ (where $\theta_{0}$ is as in \cref{prop:simpleUB}),
\begin{align*}
\max_{\theta'\in[0,\theta]}\left|\frac{\partial}{\partial\theta}\widetilde{a}_{\theta;k,\ell}^{\eps}\right|&\le C|\log\eps|^{-\frac{1}{2}}|\widehat{\rho}(\eps k)|^{2}\left(\mathbb{E}_{\mathbf{X}^{t,\mathbf{x}}}\mathscr{A}_{k,\ell}[\mathbf{X}]^{2}\right)^{1/2}\\
&\le C|\log\eps|^{-\frac{1}{2}}|\widehat{\rho}(\eps k)|^{2}|k|^{-2}.
\end{align*}
Therefore, we have
\[
\widetilde{a}_{\theta;k,\ell}^{\eps}\ge(c-C\theta)|\log\eps|^{-\frac{1}{2}}|k|^{-2}|\widehat{\rho}(\eps k)|^{2}.
\]
So as long as $\theta<\frac{c}{2C}$, we have
\[
\widetilde{a}_{\theta;k,\ell}^{\eps}\ge\frac{c}{2}|\log\eps|^{-\frac{1}{2}}|\widehat{\rho}(\eps k)|^{2}|k|^{-2},
\]
which implies \cref{eq:Alowerbound} in light of \cref{eq:QQtilde}.
\end{proof}
\begin{proof}[Proof of Theorem \ref{thm:fdoesntgoto0}.]
Define 
\[
\mathbf{Z}_{+}^{2}=\{k=(k_{1},k_{2})\in\mathbf{Z}^{2}: k_{1}>0\text{ or }(k_{1}=0\text{ and }k_{2}>0)\},
\]
so that if $k\ne k'\in\mathbf{Z}_{+}^{2}$ then $k\not\in\{k',-k'\}$, and so
\[
\int_{0}^{1}\int\e_{k,\ell}(t,x)\e_{k',\ell'}(t,x)\,\dif x\,\dif t=0
\]
and
\[
\int_{0}^{1}\int\e_{k,\ell}(t,x)\overline{\e_{k',\ell'}(t,x)}\,\dif x\,\dif t=\int_{0}^{1}\int\e_{k,\ell}(t,x)\e_{-k',-\ell'}(t,x)\,\dif x\,\dif t=0.
\]
By the discussion surrounding \cref{eq:pseudocovariance}, this  means that the random variables $\{W_{1}[\e_{k,\ell}]: k\in\mathbf{Z}_{+}^{2},\ell\in\mathbf{N}\}$
are independent. Therefore, the set 
\[
\{ W_{1}[\e_{k,\ell}]\overline{W_{1}[\e_{k,\ell}]}-1\}_{k\in\mathbf{Z}_{+}^{2},\ell\in\mathbf{N}}
\]
is an $L^2$-orthogonal collection of complex random variables. It is easy to verify that these variables have $L^2$ norm $1$, and therefore this set is actually orthonormal. Let $k_0$ be as in \cref{lem:abound}, and define
\begin{equation}
A_{\theta}^{\eps}=\sum_{\substack{k\in\mathbf{Z}_{+}^{2}\\
|k|\ge k_{0}
}
}\sum_{\ell=0}^{|k|^{2}-1}a_{\theta;k,\ell}^{\eps}(W_{1}[\e_{k,\ell}]\overline{W_{1}[\e_{k,\ell}]}-1),\label{eq:Aepsthetadef}
\end{equation}
\[
B_{\theta}^{\eps}=\int f_{\theta}^{\eps}(1,x)\,\dif x-A_{\theta}^{\eps},
\]
and
\[
E_{\theta}^{\eps}=\int h_{\theta}^{\eps}(1,x)\,\dif x-A_{\theta}^{\eps}.
\]
Then, by the orthonormality of $\{ W_{1}[\e_{k,\ell}]\overline{W_{1}[\e_{k,\ell}]}-1\}$, along with
 \cref{eq:Alowerbound}, we have
\begin{align*}
\mathbf{E}[(A_{\theta}^{\eps})^{2}]=\sum_{\substack{k\in\mathbf{Z}_{+}^{2}\\
|k|\ge k_{0}
}
}\sum_{\ell=0}^{|k|^{2}-1}(a_{\theta;k,\ell}^{\eps})^{2} & \ge c\sum_{\substack{k\in\mathbf{Z}_{+}^{2}\\
|k|\ge k_{0}
}
}\sum_{\ell=0}^{|k|^{2}-1}|\log\eps|^{-1}\theta^{2}|k|^{-4}|\widehat{\rho}(\eps k)|^{4}\\
 & =c\theta^{2}|\log\eps|^{-1}\sum_{\substack{k\in\mathbf{Z}_{+}^{2}\\
|k|\ge k_{0}
}
}|k|^{-2}|\widehat{\rho}(\eps k)|^{4}.
\end{align*}
Now, there is a $\delta>0$ so that $|\widehat{\rho}(\xi)|^{4}\ge|\widehat{\rho}(0)|^{4}/2$
whenever $|\xi|\le\delta$, so we have, as long as $\eps<\frac{\delta}{2k_{0}}$,
\begin{align*}
\mathbf{E}(A_{\theta}^{\eps})^{2}&\ge c\theta^{2}|\log\eps|^{-1}\sum_{\substack{k\in\mathbf{Z}_{+}^{2}\\
k_{0}\le|k|\le\delta/\eps
}
}|k|^{-2}|\widehat{\rho}(\eps k)|^{4}\\
&\ge c\theta^{2}|\log\eps|^{-1}\frac{|\widehat{\rho}(0)|^{4}}{2}\sum_{\substack{k\in\mathbf{Z}_{+}^{2}\\
k_{0}\le|k|\le\delta/\eps
}
}|k|^{-2}\ge c'\theta^{2}
\end{align*}
for some constant $c'>0$ depending on $\rho$. Directly from \cref{eq:Aepsthetadef} we have have $\mathbf{E}A_{\theta}^{\eps}=0.$
Furthermore, it is easy to see by \cref{eq:hvexp0} and \cref{eq:adef} that $A_{\theta}^{\eps}$ is an orthogonal projection
of $\int f_{\theta}^{\eps}(1,x)\,\dif x$ and also of $\int h_{\theta}^{\eps}(1,x)\,\dif x$ on to the $L^2$-subspace spanned by 
\[
\{ W_{1}[\e_{k,\ell}]\overline{W_{1}[\e_{k,\ell}]}-1: k\in\mathbf{Z}^2_+, \, |k|\ge k_0,\, 0\le \ell\le |k|^2-1\}.
\]
Thus, we have
\[
\mathbf{E}A_{\theta}^{\eps}B_{\theta}^{\eps}=\mathbf{E}A_{\theta}^{\eps}E_{\theta}^{\eps}=0.
\]
Recall that $A_{\theta}^{\eps}$ is a sum of squares of Gaussian random
variables, minus their expectation. (That is, it is a homogeneous element of the second
Wiener chaos.) A well-known fact about the sums of squares of Gaussian random variables (which is a special case of Gaussian hypercontractivity; see e.g. \cite[Theorem 3.50]{Jan97}) is that their higher central moments are all controlled by their variance. More precisely, for any $p>2$, there is a $C_{p}<\infty$ so that
\begin{align*}
(\mathbf{E}|A_{\theta}^{\eps}|^{p})^{1/p}&\le C_{p}(\mathbf{E}|A_{\theta}^{\eps}|^{2})^{1/2}\le C_{p}\biggl[\mathbf{E}\biggl(\int f_{\theta}^{\eps}(1,x)\,\dif x\biggr)^{2}\biggr]^{1/2}\\
&\le C_{p}\biggl[\mathbf{E}\biggl(\int h_{\theta}^{\eps}(1,x)\,\dif x\biggr)^{2}\biggr]^{1/2}\le CC_{p},
\end{align*}
where the last inequality is by \cref{eq:final-varbound-1}. Also by
\cref{eq:final-varbound-1}, we have that
\[
\mathbf{E}|B_{\theta}^{\eps}|^{2}\le\mathbf{E}\left(\int f_{\theta}^{\eps}(1,x)\,\dif x\right)^{2}\le\mathbf{E}\left(\int h_{\theta}^{\eps}(1,x)\,\dif x\right)^{2}\le C,
\]
and that
\[
\mathbf{E}|E_{\theta}^{\eps}|^{2}\le\mathbf{E}\left(\int h_{\theta}^{\eps}(1,x)\,\dif x\right)^{2}\le C.
\]
Therefore, the hypotheses of \cref{lem:nontrivialitycondition} are
satisfied with $A=A^\eps_\theta$, $B=B^\eps_\theta$ and also with $A=A^\eps_\theta$, $B=E^\eps_\theta$. Thus, neither $\int f_{\theta}^{\eps}(1,x)\,\dif x$ nor $\int h_{\theta}^{\eps}(1,x)\,\dif x$ can converge
to $0$ in probability along any subsequence.
\end{proof}

\section{Technical proofs\label{sec:techproofs}}
In this section we prove the technical lemmas which have been stated without proof earlier. 
\subsection{Derivative computations\label{subsec:derivcomps}}
Here  we give the proofs of the lemmas from \cref{sec:derivatives}.
\begin{proof}[Proof of Lemma \ref{lem:Pderivs}.]
First we prove \cref{eq:ddtheta}. For simplicity of notation, we will use the following abbreviations throughout this proof:
\[
\mathbb{E} = \mathbb{E}_{\mathbf{X}^{t,\mathbf{x}}}, \ \ \ \ \  \widehat{\mathbb{E}} = \widehat{\mathbb{E}}_{\mathbf{X}^{t,\mathbf{x}}, \widetilde{\mathbf{X}}^{t,\mathbf{x}}}^{\theta, W, \eps}, \ \ \ \ \ \mathscrbf{E}[W,\mathbf{X}] = \mathscrbf{E}_{\theta,t}^{\eps}[W,\mathbf{X}].
\] (We will sometimes use $\widehat{\mathbb{E}}$ on an expression in which only $\mathbf{X}$ appears, in which case it will be the same as if we had defined $\widehat{\mathbb{E}} = \widehat{\mathbb{E}}_{\mathbf{X}^{t,\mathbf{x}}}^{\theta, W, \eps}$.)
Let
\[
\mathscr{R}[W,\mathbf{X}] = \frac{1}{2}(\theta|\log\eps|)^{-\frac{1}{2}}\sum_{k=1}^{J}\int_{0}^{t}\int \rho^{\eps}(X_{k}(s)-y)W(\dif y\,\dif s).
\]
It is an immediate consequence of
\cref{eq:dbetaFunnyE} that
\begin{align}
\frac{\partial}{\partial\theta}\mathscrbf{E}[W,\mathbf{X}]&=\biggl(\mathscr{R}[W,\mathbf{X}] -J\frac{\partial\kappa_{\theta}^{\eps}}{\partial\theta}(t)\biggr)\mathscrbf{E}[W,\mathbf{X}].\label{eq:dbetaEE}
\end{align}
We then compute, using \cref{eq:dbetaEE}:
\[
\frac{\partial}{\partial\theta}\mathbb{E}(  \mathscr{Q}[\mathbf{X}]\mathscrbf{E}[W,\mathbf{X}]) =\mathbb{E}\biggl[\mathscr{Q}[\mathbf{X}]\biggl(\mathscr{R}[W,\mathbf{X}] -J\frac{\partial\kappa_{\theta}^{\eps}}{\partial\theta}(t)\biggr)\mathscrbf{E}[W,\mathbf{X}]\biggr].
\]
This shows that
\begin{align*}
\frac{\frac{\partial}{\partial\theta}\mathbb{E}(  \mathscr{Q}[\mathbf{X}]\mathscrbf{E}[W,\mathbf{X}])}{\mathbb{E}\mathscrbf{E}[W,\mathbf{X}]} &= \widehat{\mathbb{E}}\biggl[\mathscr{Q}[\mathbf{X}]\biggl(\mathscr{R}[W,\mathbf{X}] -J\frac{\partial\kappa_{\theta}^{\eps}}{\partial\theta}(t)\biggr)\biggr].
\end{align*}
Thus,
\begin{align*}
&\frac{\partial}{\partial\theta}\widehat{\mathbb{E}}\mathscr{Q}[\mathbf{X}] = \frac{\frac{\partial}{\partial\theta}\mathbb{E}(  \mathscr{Q}[\mathbf{X}]\mathscrbf{E}[W,\mathbf{X}])}{\mathbb{E}\mathscrbf{E}[W,\mathbf{X}]} - \frac{\mathbb{E}(  \mathscr{Q}[\mathbf{X}]\mathscrbf{E}[W,\mathbf{X}])}{\mathbb{E}\mathscrbf{E}[W,\mathbf{X}]}\frac{\frac{\partial}{\partial\theta}\mathbb{E}( \mathscrbf{E}[W,\mathbf{X}])}{\mathbb{E}\mathscrbf{E}[W,\mathbf{X}]}\\
&= \widehat{\mathbb{E}}\biggl[\mathscr{Q}[\mathbf{X}]\biggl(\mathscr{R}[W,\mathbf{X}] -J\frac{\partial\kappa_{\theta}^{\eps}}{\partial\theta}(t)\biggr)\biggr]  - \widehat{\mathbb{E}}(\mathscr{Q}[\mathbf{X}])\widehat{\mathbb{E}}\biggl(\mathscr{R}[W,\mathbf{X}] -J\frac{\partial\kappa_{\theta}^{\eps}}{\partial\theta}(t)\biggr)\\
&= \widehat{\mathbb{E}}[\mathscr{Q}[\mathbf{X}](\mathscr{R}[W,\mathbf{X}] -\mathscr{R}[W,\widetilde{\mathbf{X}}])].
\end{align*}
This completes the proof of \cref{eq:ddtheta}. The proof of \cref{eq:malliavinderiv-noE} is similar. Using \cref{eq:detaE}, we compute
\[
\Dif_{s,y}\mathscrbf{E}[W,\mathbf{X}]=\theta^{\frac{1}{2}}|\log\eps|^{-\frac{1}{2}}\sum_{k=1}^{J}\mathscr{Q}[\mathbf{X}]\rho^{\eps}(X_{k}(s)-y)\mathscrbf{E}[W,\mathbf{X}].
\]
The quotient rule then gives us
\begin{multline*}
\Dif_{s,y}\frac{\mathscrbf{E}[W,\mathbf{X}]}{\mathbb{E}\mathscrbf{E}[W,\mathbf{X}]}=\theta^{\frac{1}{2}}|\log\eps|^{-\frac{1}{2}}\sum_{k=1}^{J}\frac{\rho^{\eps}(X_{k}(s)-y)\mathscrbf{E}[W,\mathbf{X}]}{\mathbb{E}\mathscrbf{E}[W,\mathbf{X}]}\\
-\theta^{\frac{1}{2}}|\log\eps|^{-\frac{1}{2}}\sum_{k=1}^{J}\frac{\mathscrbf{E}[W,\mathbf{X}]\mathbb{E}(\rho^{\eps}(X_{k}(s)-y)\mathscrbf{E}[W,\mathbf{X}])}{(\mathbb{E}\mathscrbf{E}[W,\mathbf{X}])^{2}},
\end{multline*}
which is equation \cref{eq:malliavinderiv-noE}. Multiplying \cref{eq:malliavinderiv-noE}
by $\mathscr{Q}[\mathbf{X}]$ and taking the expectation yields \cref{eq:malliavinderiv}.
\end{proof}
\begin{proof}[Proof of Lemma \ref{lem:fullexpexpansion}.]
Define 
\[
\mathscr{H}_{k,s,y}[\mathbf{X}]=\rho^{\eps}(X_{k}(s)-y).
\]
Let $\mathbf{X}_{*}=(\mathbf{X},\widetilde{\text{\textbf{X}}},\widetilde{\widetilde{\mathbf{X}}})$
and $\mathbf{x}_{*}=(\mathbf{x},\mathbf{x},\mathbf{x})$. As in the proof of \cref{lem:Pderivs}, we abbreviate for the sake of convenience
\[
 \widehat{\mathbb{E}} = \widehat{\mathbb{E}}_{\mathbf{X}_*^{t,\mathbf{x}_*}}^{\theta, W, \eps},
\]
which, again as in the proof of \cref{lem:Pderivs}, reduces to $\widehat{\mathbb{E}} = \widehat{\mathbb{E}}_{\mathbf{X}^{t,\mathbf{x}}}^{\theta, W, \eps}$ in contexts in which only $\mathbf{X}$ appears.
Using \cref{lem:Pderivs}, we get
\begin{align*}
\frac{\partial}{\partial\theta}&\mathbf{E} \widehat{\mathbb{E}}\mathscr{Q}[\mathbf{X}]= \frac{1}{2}(\theta|\log\eps|)^{-\frac{1}{2}}\mathbf{E}\biggl[\int_{0}^{t}\int\sum_{k=1}^{J}\widehat{\mathbb{E}}(\mathscr{Q}[\mathbf{X}](\mathscr{H}_{k,s,y}[\mathbf{X}]\\
&\qquad \qquad \qquad \qquad \qquad \qquad -\mathscr{H}_{k,s,y}[\widetilde{\mathbf{X}}]))W(\dif y\,\dif s)\biggr]\\
&= \frac{1}{2}(\theta|\log\eps|)^{-\frac{1}{2}}\mathbb{E}_{\mathbf{X}^{t,\mathbf{x}}}\mathscr{Q}[\mathbf{X}]\mathbf{E}\biggl[\frac{\mathscrbf{E}_{\theta,t}^{\eps}[W,\mathbf{X}]}{\mathbb{E}_{\mathbf{X}^{t,\mathbf{x}}}\mathscrbf{E}_{\theta,t}^{\eps}[W,\mathbf{X}]}\int_{0}^{t}\int\sum_{k=1}^{J}(\mathscr{H}_{k,s,y}[\mathbf{X}]\\
&\qquad \qquad \qquad \qquad \qquad \qquad -\mathscr{H}_{k,s,y}[\widetilde{\mathbf{X}}])W(\dif y\,\dif s)\biggr]
\end{align*}
By the Gaussian integration by parts formula \cref{eq:GIP}, the above expression equals
\begin{align*}
\mathbf{E}\biggl[\int_{0}^{t}\int \sum_{k=1}^{J}\Dif_{s,y}\widehat{\mathbb{E}}(\mathscr{Q}[\mathbf{X}](\mathscr{H}_{k,s,y}[\mathbf{X}]-\mathscr{H}_{k,s,y}[\widetilde{\mathbf{X}}]))\,\dif y\,\dif s\biggr].
\end{align*}
It is not difficult to see by \cref{eq:malliavinderiv} and symmetry considerations that 
\begin{align*}
&\Dif_{s,y}\widehat{\mathbb{E}}(\mathscr{Q}[\mathbf{X}](\mathscr{H}_{k,s,y}^{\eps}[\mathbf{X}]-\mathscr{H}_{k,s,y}^{\eps}[\widetilde{\mathbf{X}}]))\\
&= \theta^{\frac{1}{2}}|\log\eps|^{-\frac{1}{2}}\sum_{\ell=1}^{J}\widehat{\mathbb{E}}(\mathscr{Q}[\mathbf{X}] A_{k,\ell}(s,y)),
\end{align*}
where
\begin{align*}
A_{k,\ell}(s,y) &= (\mathscr{H}_{k,s,y}[\mathbf{X}]-\mathscr{H}_{k,s,y}[\widetilde{\mathbf{X}}])(\mathscr{H}_{\ell,s,y}[\mathbf{X}]-\mathscr{H}_{\ell,s,y}[\widetilde{\widetilde{\mathbf{X}}}])\\
 & \qquad\qquad+(\mathscr{H}_{k,s,y}[\mathbf{X}]-\mathscr{H}_{k,s,y}[\widetilde{\mathbf{X}}])(\mathscr{H}_{\ell,s,y}[\widetilde{\mathbf{X}}]-\mathscr{H}_{\ell,s,y}[\widetilde{\widetilde{\mathbf{X}}}]).
\end{align*}
By \cref{eq:itform}, we get
\begin{align*}
&\int_0^t \int A_{k,\ell}(s,y)\, \dif y\, \dif s\\
&= \mathscr{I}_{t}^{\eps}[X_{k},X_{\ell}]-\mathscr{I}_{t}^{\eps}[\widetilde{X}_{k},X_{\ell}]-\mathscr{I}_{t}^{\eps}[X_{k},\widetilde{\widetilde{X}}_{\ell}]+\mathscr{I}_{t}^{\eps}[\widetilde{X}_{k},\widetilde{\widetilde{X}}_{\ell}]\\
 & \qquad+\mathscr{I}_{t}^{\eps}[X_{k},\widetilde{X}_{\ell}]-\mathscr{I}_{t}^{\eps}[\widetilde{X}_{k},\widetilde{X}_{\ell}]-\mathscr{I}_{t}^{\eps}[X_k,\widetilde{\widetilde{X}}_{\ell}]+\mathscr{I}_{t}^{\eps}[\widetilde{X}_{k},\widetilde{\widetilde{X}}_{\ell}].
 \end{align*}
By symmetry,
\begin{align*}
\mathbf{E}(\mathscr{Q}[\mathbf{X}] \mathscr{I}_{t}^{\eps}[X_{k},\widetilde{X}_{\ell}]) &= \mathbf{E}(\mathscr{Q}[\mathbf{X}] \mathscr{I}_{t}^{\eps}[X_{k},\widetilde{\widetilde{X}}_{\ell}]).
\end{align*}
With this simplification, we get
\begin{align*}
&\mathbf{E}\biggl(\mathscr{Q}[\mathbf{X}]\int_0^t \int A_{k,\ell}(s,y)\, \dif y\, \dif s\biggr)\\
&=\mathbf{E}(\mathscr{Q}[\mathbf{X}] (\mathscr{I}_{t}^{\eps}[X_{k},X_{\ell}]- \mathscr{I}_{t}^{\eps}[X_{k}, \widetilde{X}_{\ell}]- \mathscr{I}_{t}^{\eps}[\widetilde{X}_{k}, X_{\ell}] \\
&\qquad \qquad - \mathscr{I}_{t}^{\eps}[\widetilde{X}_{k}, \widetilde{X}_{\ell}]+2 \mathscr{I}_{t}^{\eps}[\widetilde{X}_{k},\widetilde{\widetilde{X}}_{\ell}])).
\end{align*}
Since $\mathscr{I}_t^\eps$ is symmetric in its arguments, this shows that
\begin{align*}
&\sum_{k,\ell=1}^J\mathbf{E}\biggl(\mathscr{Q}[\mathbf{X}]\int_0^t \int A_{k,\ell}(s,y)\, \dif y\, \dif s\biggr)\\
&=\sum_{k,\ell=1}^J\mathbf{E}(\mathscr{Q}[\mathbf{X}] (\mathscr{I}_{t}^{\eps}[X_{k},X_{\ell}]- 2\mathscr{I}_{t}^{\eps}[X_{k}, \widetilde{X}_{\ell}] \\
&\qquad \qquad - \mathscr{I}_{t}^{\eps}[\widetilde{X}_{k}, \widetilde{X}_{\ell}]+2 \mathscr{I}_{t}^{\eps}[\widetilde{X}_{k},\widetilde{\widetilde{X}}_{\ell}])).
\end{align*}
If $k=\ell$, then \cref{eq:itsame} implies that $\mathscr{I}_{t}^{\eps}[X_{k},X_{\ell}] = \mathscr{I}_{t}^{\eps}[\widetilde{X}_{k}, \widetilde{X}_{\ell}]$. On the other hand, if $k\ne \ell$, then 
\[
\mathbf{E}(\mathscr{Q}[\mathbf{X}]\mathscr{I}_{t}^{\eps}[\widetilde{X}_{k}, \widetilde{X}_{\ell}]) = \mathbf{E}(\mathscr{Q}[\mathbf{X}]\mathscr{I}_{t}^{\eps}[\widetilde{X}_{k}, \widetilde{\widetilde{X}}_{\ell}]).
\]
This shows that 
\begin{align*}
&\sum_{k,\ell=1}^J\mathbf{E}\biggl(\mathscr{Q}[\mathbf{X}]\int_0^t \int A_{k,\ell}(s,y)\, \dif y\, \dif s\biggr)\\
&=\sum_{k,\ell=1}^J\mathbf{E}(\mathscr{Q}[\mathbf{X}] (\mathscr{I}_{t}^{\eps}[X_{k},X_{\ell}]\mathbf{1}_{k\ne \ell}- 2\mathscr{I}_{t}^{\eps}[X_{k}, \widetilde{X}_{\ell}] \\
&\qquad \qquad + (1+\mathbf{1}_{k=\ell}) \mathscr{I}_{t}^{\eps}[\widetilde{X}_{k},\widetilde{\widetilde{X}}_{\ell}])).
\end{align*}
The proof is now easily completed by combining the above calculations.
\end{proof}

\subsection{Brownian motion computations\label{subsec:BMcomps}}

We need two preliminary lemmas.
\begin{lem}\label{lem:heat}
Let 
\[
p_{t}(x)=\frac{1}{2\pi t}\sum_{z\in\mathbf{Z}^{2}}\e^{-\frac{|x+z|^{2}}{2t}}
\]
be the periodic heat kernel. Then there is a constant $C$ so that
\[
p_{t}(x)\le C(1+t^{-1})\e^{-\frac{1}{2t}|x|_{\mathbf{T}^{2}}^{2}}.
\]
\end{lem}
\begin{proof}
 Without loss of generality, we may assume that $|x|_{\mathbf{T}^2} = |x|$. It is then sufficient to show that
\[ \sum_{z\in\mathbf{Z}^{2}\setminus\{0\}}\e^{-\frac{|x+z|^{2}}{2t}-\log t}\]
is bounded by a constant independent of $t$ and $x$. For $t\ge c$ the result is trivial, so we may assume that $t<c$ for some constant $c$ to be chosen later. We note that since $|x|=|x|_{\mathbf{T}^2}$, for all $z\in\mathbf{Z}^2\setminus\{0\}$ we have that $|z+x|\ge 1/2$, and so 
\[
 \frac{|x+z|^2}{2t}+\log t\ge  \frac{|x+z|^2}{t} \ge \frac{1}{c}|x+z|^2
\]
for $t$ sufficiently small. (Choose $c$ small enough so that this holds.) Then the result follows from the fact that
\[ 
\sum_{z\in\mathbf{Z}^{2}\setminus\{0\}}\e^{-\frac{1}{c}|x+z|^{2}}<\infty, \]
which is a simple exercise.
\end{proof}

 \begin{lem}\label{lem:rlemma}
There is an absolute constant $C$ such that for any $z\in \mathbf{T}^2$ and any $\eps<1/4$, 
\[
\left|\int R^{\eps}(w)\log|w-z|_{\mathbf{T}^{2}}^{-2}\,\dif w\right|\le C\left(1+\log|z|_{\mathbf{T}^{2}}^{-2}\right).
\]
\end{lem}

\begin{proof}
Consider $R^\eps$ as a function on $\mathbf{R}^2$, by identifying $\mathbf{T}^2$ with $(-1/2,1/2)^2$ and defining $R^\eps$ to be zero outside this square. When $\eps< 1/4$, it is not difficult to see that  $R^\eps(x)=0$ for all $|x|>\sqrt{2}\eps$. Also, there is an absolute constant $c$ such that any $w$ with $|w|<\sqrt{2}/4$ satisfies $|w-z|_{\mathbf{T}^2} \ge c|w-z|$ for all $z\in \mathbf{T}^2$. Thus, when $\eps<1/4$, it suffices to show that
\begin{equation}
\int_{\mathbf{R}^{2}}R^{\eps}(w)\log|w-z|^{-2}\,\dif w\le C\left(1+\log|z|^{-2}\right).\label{eq:logbound}
\end{equation}
First, suppose that $|z|\le 4\eps$. In this situation, if $|w-z|> 6\eps$, then $|w|>2\eps$ and hence $R^\eps(w)=0$. Thus,
\begin{align}
\int R^{\eps}(w)\log|w-z|^{-2}\,\dif w & \le\|R^\eps\|_{L^{\infty}}\int_{\{|w-z|\le6\eps\}}\log|w-z|^{-2}\,\dif w\nonumber \\
 & \le C \eps^{-2}\int_{\{|w|\le6\eps\}}\log|w|^{-2}\,\dif w\nonumber \\
 & =C\eps^{-2}\int_{0}^{6\eps}(-2r\log r)\,\dif r\nonumber \\
 &\le C\log \eps^{-1}\le C(1+\log |z|^{-2}).\label{eq:littlez}
\end{align}
On the other hand, suppose that $|z|>4\eps$. In this situation, if $|w-z|\le |z|/2$, then $|w|\ge |z|-|w-z|\ge |z|/2>2\eps$, and hence $R(w)=0$. Thus, 
\begin{align}
\int R^{\eps}(w)\log|w-z|^{-2}\,\dif w & \le\int R^{\eps}(w)\log(|z|/2)^{-2}\,\dif w \nonumber\\
&\le C(1+\log|z|^{-2}).\label{eq:bigz}
\end{align}
The two bounds \cref{eq:littlez} and \cref{eq:bigz} together imply \cref{eq:logbound}.
\end{proof}

\begin{proof}[Proof of Lemma \ref{lem:logupperbound}.]
Define the integration domain 
\[
\Xi^{r}(t)=\{(s_1,\ldots,s_r):0\le  s_{1}\le\cdots\le s_{r}\le t\}
\]
and put $Z=X-Y$. Note that
\begin{align*}
&\mathbb{E}_{X^{t,x}}\mathscr{I}_{t}^{\eps}[X,Y]^{r}=\mathbb{E}_{X^{t,x}}\biggl[\int_{[0,t]^{r}}\prod_{i=1}^{r}R^{\eps}(Z(s_{i}))\,\dif s_{1}\cdots\dif s_{r}\biggr]\\
&=r!\int_{\Xi^{r}(t)}\mathbb{E}_{X^{t,x}}\biggl(\prod_{i=1}^{r}R^{\eps}(Z(s_{i}))\,\dif s_{1}\cdots\dif s_{r}\biggr)\\
&=r!\int_{\Xi^{r}(t)}\mathbb{E}_{X^{t,x}}\biggl(\prod_{i=1}^{r}R^{\eps}(X(s_{i})-X(s_{i+1})+X(s_{i+1})- Y(s_i))\biggr)\,\dif s_{1}\cdots\dif s_{r},
\end{align*}
where we set $s_{r+1} = t$. 
Recalling that $X$ is a Brownian motion running backward in time, and the independent increments property of Brownian motion, it is easy to show by backward induction that the last integral is bounded by
\begin{align*}
&\int_{\Xi^{r}(t)}\max_{z_{1},\ldots,z_{r}\in\mathbf{T}^{2}}\prod_{i=1}^{r}\mathbb{E}_{X^{t,x}}[R^{\eps}(X(s_{i})-X(s_{i+1})+z_i)]\,\dif s_{1}\cdots\dif s_{r}.
\end{align*}
In the following we will use the notations $a\wedge b$ and $a\vee b$ to denote the minimum and maximum of $a$ and $b$, respectively. By \cref{lem:heat}, 
\begin{align}
&\mathbb{E}_{X^{t,x}}[R^{\eps}(X(s_{i})-X(s_{i+1})+z)]=\int R^{\eps}(w+z)p_{s_{i+1}-s_i}(w)\,\dif w\notag\\
&\le (\|R^\eps\|_{L^1}\|p_{s_{i+1}-s_i}\|_{L^\infty})\wedge (\|R^\eps\|_{L^\infty}\|p_{s_{i+1}-s_i}\|_{L^1})\notag\\
&\le C\left(1+\frac{1}{s_{i}-s_{i-1}}\right)\wedge(\|\rho^\eps\|_{L^1}\|\rho^\eps\|_{L^\infty})\notag\\
&\le\frac{C}{((s_{i}-s_{i-1})\wedge1)\vee \eps^2}.\label{eq:ER}
\end{align}
Thus there is a constant $C$ so that, if we define the integration
domain 
\[
\Delta^{r}(t)=\{(s_{1},\ldots,s_{r})\in\mathbf{R}_{\ge 0}^r: s_{1}+\cdots+s_{r}\le t\},
\]
then
\begin{align}
\mathbb{E}_{X^{t,x}}\mathscr{I}_{t}^{\eps}[X,Y]^{r}&\le C^{r}r!\int_{\Xi^{r}(t)}\prod_{i=1}^{r}\frac{1}{((s_{i}-s_{i-1})\wedge1)\vee\eps^{2}}\,\dif s_{1}\cdots\dif s_{r}\notag\\
&=C^{r}r!\int_{\Delta^{r}(t)}\prod_{i=1}^{r}\frac{1}{(u_{i}\wedge1)\vee\eps^{2}}\,\dif u_{1}\cdots\dif u_{r}\notag\\
&\le C^{r}r!\left(\int_{0}^{t}\frac{1}{(u\wedge1)\vee\eps^{2}}\,\dif u \right)^{r}\notag\\
&\le C^{r}r!\left(t+\log\eps^{-2}\right)^{r}.\label{eq:calc}
\end{align}
The statement of the lemma follows by the assumption that $\log\eps^{-2}\ge t$.
\end{proof}
\begin{proof}[Proof of Lemma \ref{lem:X0Yest}.]
We first prove \cref{eq:IX0Ymoments} for $r=1$. In this case we have,
using \cref{lem:heat},
\begin{align}
&\mathbb{E}_{X_{1}^{t,x_{1}},X_{2}^{t,x_{2}}}\biggl(\int_{0}^{t}R^{\eps}(X_{1}(s)-X_{2}(s))\,\dif s\biggr)\notag\\
&=\int_{0}^{t}\int R^{\eps}(w)p_{2(t-s)}(w-(x_{1}-x_{2}))\,\dif w\,\dif s\notag\\
&\le C\int R^{\eps}(w)\int_{0}^{t}\biggl(1+\frac{1}{2(t-s)}\biggr)\e^{-\frac{1}{4(t-s)}|w-(x_{1}-x_{2})|_{\mathbf{T}^{2}}^{2}}\,\dif s\,\dif w.\notag
\end{align}
Applying a change of variable to the inner integral, and noting that $\|R^\eps\|_{L^1} = \|\rho^\eps\|_{L^1}^2=\|\rho\|_{L^1}^2$ does not depend on $\eps$, we see  that the above quantity is bounded by
\begin{align*}
&C\left(t+\int R^{\eps}(w)\int_{0}^{t|w-(x_{1}-x_{2})|_{\mathbf{T}^{2}}^{-2}}\frac{1}{2s}\e^{-\frac{1}{4s}}\,\dif s\,\dif w\right)\\
&\le C\left(t+\int R^{\eps}(w)\int_{0}^{(t+1)|w-(x_{1}-x_{2})|_{\mathbf{T}^{2}}^{-2}}\frac{1}{2s}\e^{-\frac{1}{4s}}\,\dif s\,\dif w\right)\\
&\le C\left(t+\log (t+1) + \int R^{\eps}(w)\log|w-(x_{1}-x_{2})|_{\mathbf{T}^{2}}^{-2}\,\dif w\right)\\
&\le  C\left(t+1 + \int R^{\eps}(w)\log|w-(x_{1}-x_{2})|_{\mathbf{T}^{2}}^{-2}\,\dif w\right)
\end{align*}
Thus, by \cref{lem:rlemma},
\begin{equation}
\mathbb{E}_{X_{1}^{t,x_{1}},X_{2}^{t,x_{2}}}\biggl(\int_{0}^{t}R^{\eps}(X_{1}(s)-X_{2}(s))\,\dif s\biggr)
\le C(t+1 + \log|x_{1}-x_{2}|_{\mathbf{T}^{2}}^{-2}).\label{eq:farapartr1}
\end{equation}
Now we can estimate the general case. Let $Z=X_{1}-X_{2}$ and abbreviate
$\mathbb{E}=\mathbb{E}_{X_{1}^{t,x_{1}},X_{2}^{t,x_{2}}}$. Then 
\begin{align*}
 & \mathbb{E}\left(\int_{0}^{t}R^{\eps}(Z(s))\,\dif s\right)^{r}\\
 &=\int_{[0,t]^{r}}\mathbb{E}\biggl[\prod_{i=1}^{r}R^{\eps}(Z(s_{i}))\biggr]\,\dif s_{1}\cdots\,\dif s_{r}\\
 & =r!\int_{\Xi^{r}(t)}\mathbb{E}\biggl[\prod_{i=1}^{r}R^{\eps}(Z(s_{i})-Z(s_{i+1})+Z(s_{i+1}))\biggr]\,\dif s_{1}\cdots\,\dif s_{r},
 \end{align*}
 where, as before, we use the convention $s_{r+1}=t$. 
By the independent increments property of Brownian paths (flowing backward in time) and the bounds \cref{eq:ER}, \cref{eq:calc} and \cref{eq:farapartr1}, the last integral is bounded by 
\begin{align*}
 & \int_{\Xi^{r}(t)}\max_{\substack{z_1,\ldots,z_{r-1}\in\mathbf{T}^{2}\\ z_r = x_1-x_2}}\prod_{i=1}^{r}\mathbb{E}[R^{\eps}(Z(s_{i})-Z(s_{i+1})+z_i)]\,\dif s_{1}\cdots\,\dif s_{r}\\
 & \le C^{r}\int_{\Xi^{r}(t)}\left(t+1+\log|x_{1}-x_{2}|_{\mathbf{T}^{2}}^{-2}\right)\\
 &\qquad \qquad \cdot\prod_{i=1}^{r-1}(((s_{i}-s_{i-1})\wedge1)\vee\eps^{2})^{-1}\,\dif s_{1}\cdots\,\dif s_{r}\\
 & \le C^{r}(t+1+\log|x_{1}-x_{2}|_{\mathbf{T}^{2}}^{-2})(t+\log\eps^{-2})^{r-1}.
\end{align*}
The assumption that $t\le\log\eps^{-2}$ completes the argument.
\end{proof}

\begin{proof}[Proof of Lemma \ref{lem:Mp}.]
Let $\Xi^{2p} = \Xi^{2p}(1)$, where $\Xi^{2p}(1)$ is defined as in the proof of \cref{lem:logupperbound} above. Similarly, let $\Delta^{2p} = \Delta^{2p}(1)$. Let $Z_{1},\ldots,Z_{m}$ be i.i.d.~standard Gaussian random variables, and let 
\[
Q_{2p}=\biggl\{\alpha\in\{-1,1\}^{2p}:\sum_{j=1}^{2p}\alpha_{j}=0\biggr\}.
\] 
Using \cref{eq:brownian}, we can expand and integrate $M_{k,\ell;2p}$ as
\begin{align}
M_{k,\ell;2p} & = (2p)!\mathbb{E}_{X^{1,0}}\biggl[\int\int_{\Xi^{2p}}\sum_{\alpha\in Q_{2p}}\prod_{n=1}^{2p}\exp\{2\pi\alpha_{n}\ii(\ell s_{n}\notag \\
&\qquad \qquad \qquad \qquad +k\cdot X(s_{n})+2\pi k\cdot x)\}\,\dif s_{1}\cdots\dif s_{2p}\,\dif x\biggr]\nonumber \\
 & =(2p)!\int_{\Xi^{2p}}\sum_{\substack{\alpha\in Q_{2p}}
}\mathbf{E}\prod_{n=1}^{2p}\exp\biggl\{ 2\pi\alpha_{n}\ii\biggl(\ell s_{n}\notag\\
&\qquad \qquad \qquad \qquad +|k|\sum_{m=1}^{n}\sqrt{s_{m}-s_{m-1}}Z_{m}\biggr)\biggr\} \,\dif s_{1}\cdots\dif s_{2p}.\label{eq:M2Xi}
\end{align}
Put $t_{m}=s_{m}-s_{m-1}$. Then the expectation in \cref{eq:M2Xi}
is
\begin{align}
&\mathbf{E}\prod_{n=1}^{2p}\exp\left\{ 2\pi\alpha_{n}\ii\left(\ell s_{n}+|k|\sum_{m=1}^{n}\sqrt{t_{m}}Z_{m}\right)\right\}  \notag\\
& =\exp\left\{ 2\pi\ii\ell\sum_{m=1}^{2p}\widetilde{\alpha}_{m}t_{m}\right\} \prod_{m=1}^{2p}\mathbf{E}\exp\left\{ 2\pi\ii|k|\widetilde{\alpha}_{m}\sqrt{t_{m}}Z_{m}\right\} \nonumber \\
 & =\exp\left\{ 2\pi\sum_{m=1}^{2p}\left(\ii\ell\widetilde{\alpha}_{m}-\pi|k|^{2}\widetilde{\alpha}_{m}^{2}\right)t_{m}\right\} ,\label{eq:Mp2}
\end{align}
where we use the notation 
$
\widetilde{\alpha}_{m}=\sum_{n=m}^{2p}\alpha_{n}.\label{eq:atildedef}
$
 Substituting \cref{eq:Mp2} into \cref{eq:M2Xi}, 
 we have
\begin{align}
M_{k,\ell;2p} & =(2p)!\int_{\Xi^{2p}}\sum_{\alpha\in Q_{2p}}\exp\biggl\{ 2\pi\sum_{m=1}^{2p}(\ii\ell\widetilde{\alpha}_{m}-\pi|k|^{2}\widetilde{\alpha}_{m}^{2})\notag \\
&\qquad \qquad \qquad \qquad \qquad \cdot(s_{m}-s_{m-1})\biggr\} \,\dif s_{1}\cdots\dif s_{2p}\notag \\
&= (2p)!\int_{\Delta^{2p}}\sum_{\alpha\in Q_{2p}}\exp\biggl\{ 2\pi\sum_{m=1}^{2p}(\ii\ell\widetilde{\alpha}_{m}
-\pi|k|^{2}\widetilde{\alpha}_{m}^{2})t_{m}\biggr\} \,\dif t_{1}\cdots\dif t_{2p}.\label{eq:M2plastequality}
\end{align}
Now define a modified integration domain
\[
B(\alpha)=\bigtimes_{m=1}^{2p}B_{m}(\alpha),
\]
where
\[
B_{m}(\alpha)=\begin{cases}
[0,1] & \widetilde{\alpha}_{m}=0\\{}
[0,\infty) & \text{otherwise.}
\end{cases}
\]
Noting that $\Delta^{2p}\subset B(\alpha)$ for each $\alpha$, we 
estimate
\begin{align*}
|M_{k,\ell;2p}|&\le(2p)!\sum_{\alpha\in Q_{2p}}\int_{B(\alpha)}\exp\biggl\{ -2\pi^2|k|^{2}\sum_{m=1}^{2p}\widetilde{\alpha}_{m}^{2}t_{m}\biggr\} \,\dif t_{1}\cdots\dif t_{2p}\\
&= (2p)! \sum_{\alpha\in Q_{2p}} \prod_{m\, :\, \widetilde{\alpha}_m \ne 0}  \frac{1}{2\pi^2\widetilde{\alpha}_m^2|k|^2}.
\end{align*}
Since $|\widetilde{\alpha}_{m}-\widetilde{\alpha}_{m-1}|=1$
for each $m$, we must have that $|\{m:\widetilde{\alpha}_{m}=0\}|\le p$. Moreover, each $\widetilde{\alpha}_m$ is an integer. Therefore,  we get
\begin{align*}
|M_{k,\ell;2p}|  &\le(2p)!|Q_{2p}|(2\pi^2|k|^2)^{-p} = (2p)!\binom{2p}{p}(2\pi^2|k|^2)^{-p},
\end{align*}
which proves \cref{eq:upperbound}. Now we prove \cref{eq:secondmoment}. By \cref{eq:M2plastequality} applied
when $p=1$, we have (noting that $Q_2$ has only two elements),
\begin{align}
M_{2}  &=2\int_{\Delta^{2}}(\exp\{ 2\pi(\ii\ell-\pi|k|^{2})t_{2}\} 
+\exp\{ 2\pi(-\ii\ell-\pi|k|^{2})t_{2}\})\,\dif t_{1}\,\dif t_{2}\nonumber \\
 & =2\int_{0}^{1}(1-t)(\exp\{ 2\pi(\ii\ell-\pi|k|^{2})t\} 
+\exp\{ 2\pi(-\ii\ell-\pi|k|^{2})t\} )\,\dif t.\label{eq:M2decomp}
\end{align}
Now note that
\begin{align}
&\int_{0}^{1}(\exp\{ 2\pi(\ii\ell-\pi|k|^{2})t\} +\exp\{ 2\pi(-\ii\ell-\pi|k|^{2})t\} )\,\dif t\notag\\
&=\frac{1}{2\pi}\biggl(\frac{1-\e^{-\ii\ell-\pi|k|^{2}}}{\ii\ell+\pi|k|^{2}}+\frac{1-\e^{\ii\ell-\pi|k|^{2}}}{-\ii\ell+\pi|k|^{2}}\biggr)\notag\\
&= \frac{|k|^{2}}{\ell^{2}+\pi^{2}|k|^{4}}-\frac{1}{2\pi}\biggl(\frac{\e^{-\ii\ell}}{\ii\ell+\pi|k|^{2}}+\frac{\e^{\ii\ell}}{-\ii\ell+\pi|k|^{2}}\biggr)e^{-\pi|k|^2}.\label{eq:M2UB}
\end{align}
We further have
\begin{align}
\biggl|\int_{0}^{1}t\exp\{ 2\pi(\pm\ii\ell-\pi|k|^{2})t\} \,\dif t\biggr|&\le\int_{0}^{\infty}t\exp\{ -2\pi^{2}|k|^{2}t\} \,\dif t
=\frac{1}{4\pi^{4}|k|^{4}}.\label{eq:M2LB}
\end{align}
Combining \cref{eq:M2decomp}, \cref{eq:M2UB} and \cref{eq:M2LB}, and recalling that $M_2\ge0$, we get
\begin{align*}
M_2 &= |M_2|\\
&\ge \frac{|k|^{2}}{\ell^{2}+\pi^{2}|k|^{4}}-\biggl|\frac{1}{2\pi}\biggl(\frac{\e^{-\ii\ell}}{\ii\ell+\pi|k|^{2}}+\frac{\e^{\ii\ell}}{-\ii\ell+\pi|k|^{2}}\biggr)e^{-\pi|k|^2}\biggr|-\frac{1}{4\pi^{4}|k|^{4}}\\
&\ge \frac{|k|^{2}}{\ell^{2}+\pi^{2}|k|^{4}}-e^{-\pi|k|^2}-\frac{1}{4\pi^{4}|k|^{4}}.
\end{align*}
Since $|\ell|\le |k|^2$, this proves \cref{eq:secondmoment}.
\end{proof}

\subsection{Proof of \cref{lem:nontrivialitycondition}\label{subsec:prob-proofs}}
Suppose for the sake of contradiction that $A_{n}+B_{n}$ converges
in probability to $0$. Then for each $\delta,\eta>0$ we have an
$n$ so that
\[
\mathbf{P}(|A_{n}+B_{n}|\ge\eta)<\delta.
\]
Note that
\begin{align}
0&=  \mathbf{E}A_{n}B_{n}\notag\\
&=\mathbf{E}A_{n}B_{n}\mathbf{1}\{|A_{n}+B_{n}|<\eta\}+\mathbf{E}A_{n}B_{n}\mathbf{1}\{|A_{n}+B_{n}|\ge\eta\}.\label{eq:splitupcov}
\end{align}
Now choose 
\[
\alpha=\frac{1}{1/p+1/2},\qquad\beta=\frac{\alpha}{\alpha-1},\qquad r=\frac{p}{\alpha},\qquad q=\frac{2}{\alpha},
\]
so that $(\alpha,\beta,r,q)\in(1,\infty)$, $1/\alpha+1/\beta=1/r+1/q=1$,
$r\alpha=p$, and $q\alpha=2$. By the H\"older and Young inequalities,
we then have
\begin{multline*}
\left|\mathbf{E}A_{n}B_{n}\mathbf{1}\{|A_{n}+B_{n}|\ge\eta\}\right|\le\delta^{1/\beta}(\mathbf{E}|A_{n}|^{\alpha}|B_{n}|^{\alpha})^{1/\alpha}\\
\le\delta^{1/\beta}\left(\frac{\mathbf{E}|A_{n}|^{r\alpha}}{r}+\frac{\mathbf{E}|B_{n}|^{q\alpha}}{q}\right)^{1/\alpha}
=\delta^{1/\beta}\left(\frac{\mathbf{E}|A_{n}|^{p}}{r}+\frac{\mathbf{E}|B_{n}|^{2}}{q}\right)^{1/\alpha}\\\le\delta^{1/\beta}C^{1/\alpha}.
\end{multline*}
Thus, by \cref{eq:splitupcov}, we get
\begin{equation}
\left|\mathbf{E}A_{n}B_{n}\mathbf{1}\{|A_{n}+B_{n}|<\eta\}\right|\le\delta^{1/\beta}C^{1/\alpha}.\label{eq:covUB}
\end{equation}
On the other hand, we have
\begin{align*}
&|\mathbf{E}A_{n}B_{n}\mathbf{1}\{|A_{n}+B_{n}|<\eta\}|\\
&\ge |\mathbf{E}(-A_n^2\mathbf{1}\{|A_{n}+B_{n}|<\eta\})| - |\mathbf{E}A_{n}(A_n+B_{n})\mathbf{1}\{|A_{n}+B_{n}|<\eta\}|\\
&\ge \mathbf{E}(A_n^2) - \mathbf{E}(A_n^2\mathbf{1}\{|A_{n}+B_{n}|\ge\eta\}) - \eta \mathbf{E}|A_n|.
\end{align*}
By H\"older's inequality,
\begin{align*}
\mathbf{E}(A_n^2\mathbf{1}\{|A_{n}+B_{n}|\ge\eta\}) &\le (\mathbf{E}|A_n|^p)^{2/p}(\mathbf{P}(|A_n+B_n|\ge \eta))^{(p-2)/p}\\
&\le C^{2/p}\delta^{(p-2)/p}.
\end{align*}
Also, $\mathbf{E}|A_n|\le C^{1/p}$. Combining the last three displays, we get
\begin{align}
|\mathbf{E}A_{n}B_{n}\mathbf{1}\{|A_{n}+B_{n}|<\eta\}| &\ge \mathbf{E}(A_{n}^{2}) - C^{2/p}\delta^{(p-2)/p} - \eta C^{1/p}.\label{eq:covLB}
\end{align} 
But, combining \cref{eq:covUB}
and \cref{eq:covLB}, we get
\[
\delta^{1/\beta}C^{1/\alpha}\ge c- C^{2/p}\delta^{(p-2)/p} - \eta C^{1/p},
\]
which is absurd once we choose $\eta$ and $\delta$ sufficiently small.
(Note that we can do this since $\alpha$ and $\beta$ depend only
on $p$ and not on $\eta$ or $\delta$.)

\end{document}